\documentclass{llncs}

\usepackage{amsmath}
\usepackage{url}
\usepackage{algorithm}
\usepackage[noend]{algorithmic}
\usepackage{dsfont}
\usepackage{mathtools}
\usepackage{color}
\usepackage{tikz}
\usepackage{graphicx}
\usepackage{subfigure}

\algsetup{indent=2em}

\newcommand{\R}{\bbbr}
\newcommand{\N}{\bbbn}
\renewcommand{\O}{\mathcal{O}}

\renewcommand{\l}{$\ell_{1}$}
\renewcommand{\Pr}{\mathcal{P}}

\newcommand{\X}{X}%{\mathcal{X}}

\renewcommand{\d}{d}%{{\rm dist}}
\newcommand{\st}{{\rm s.\,t.}}
\newcommand{\norm}[1]{\lVert {#1} \rVert}
\newcommand{\abs}[1]{\lvert {#1} \rvert}
\newcommand{\suchthat}{\; | \;}
\newcommand{\define}{\coloneqq}
\newcommand{\E}{\mathds{E}}
\DeclareMathOperator{\interior}{int}
\DeclareMathOperator{\dom}{dom}
\DeclareMathOperator{\argmin}{argmin}

\begin{document}

\title{An Infeasible-Point Subgradient Method\\Using Adaptive Approximate
  Projections\thanks{This work has been funded by the Deutsche
    Forschungsgemeinschaft (DFG) within the project ``Sparse Exact and
    Approximate Recovery'' under grants LO 1436/3-1 and PF
    709/1-1. Moreover, D.~Lorenz acknowledges support from the DFG
    project ``Sparsity and Compressed Sensing in Inverse Problems''
    under grant LO 1436/2-1.}}  
\author{\sc Dirk A. Lorenz\inst{1} \and Marc E. Pfetsch\inst{2} \and Andreas M. Tillmann\inst{2}} 
%\institute{Department of Mathematics,
%  Carl-Friedrich-Gau\ss-Faculty\\Technische Universit\"at
%  Braunschweig, 38106 Braunschweig, Germany} 
\institute{Institute for Analysis and Algebra, TU Braunschweig, Germany \and Research Group Optimization, TU Darmstadt, Germany}
\maketitle

\begin{abstract}
  We propose a new subgradient method for the minimization of
  nonsmooth convex functions over a convex set.
  To speed up computations we use adaptive approximate
  projections only requiring to move within a certain distance of the exact
  projections (which decreases in the course of the algorithm). In
  particular, the iterates in our method can be infeasible throughout
  the whole procedure. Nevertheless, we provide conditions which
  ensure convergence to an optimal feasible point under suitable
  assumptions. One convergence result deals with step size sequences
  that are fixed a priori. Two other results handle dynamic
  Polyak-type step sizes depending on a lower or upper estimate of the
  optimal objective function value, respectively. Additionally, we
  briefly sketch two applications: Optimization with convex chance
  constraints, and finding the minimum \l-norm solution to an
  underdetermined linear system, an important problem in Compressed
  Sensing.
\end{abstract}
% Keywords: Subgradient Method, Approximate Projection, Sparse Minimization, \l-norm Minimization, Compressed Sensing, ...
\begin{center}{\today}\end{center}

\section{Introduction}\label{sect:intro}
The projected subgradient method \cite{S85} is a classical algorithm
for the minimization of a nonsmooth convex function $f$ over a convex
closed constraint set $\X$, i.e., for the problem
\begin{equation}\label{eq:NLO}
\min\, f(x)\quad\st\quad x \in \X.  
\end{equation}
One iteration consists of taking a step of size $\alpha_k$ along the
negative direction of an arbitrary subgradient $h^k$ of the objective
function $f$ at the current point $x^k$ and then computing the next
iterate by projection ($\Pr_\X$) onto the feasible set~$\X$:
\[
x^{k+1} = \Pr_\X(x^k - \alpha_k\, h^k).
\]

Over the past decades, numerous extensions and specializations of this
scheme have been developed and proven to converge to a minimum (or
minimizer). Well-known disadvantages of the subgradient method are its
slow local convergence and the necessity to extensively tune
algorithmic parameters in order to obtain practical convergence. On
the positive side, subgradient methods involve fast iterations and are
easy to implement. In fact, they have been widely used in applications
and (still) form one of the most popular algorithms for nonsmooth
convex minimization.

The main effort in each iteration of the projected subgradient
algorithm usually lies in the computation of the
projection~$\Pr_\X$. Since the projection is the solution of a
(smooth) convex program itself, the required time depends on the
structure of~$\X$ and corresponding specialized algorithms. Examples
admitting a fast projection include the case where~$\X$ is the
nonnegative orthant or the \l-norm-ball $\{\,x \suchthat \norm{x}_1
\leq \tau\,\}$, onto which any $x \in \R^n$ can be projected in
$\O(n)$ time, see \cite{vdBSFM08}. The projection is more involved
if~$\X$ is, for instance, an affine space or a (convex) polyhedron. In
these latter cases, it makes sense to replace the exact
projection~$\Pr_\X$ by an approximation~$\Pr^\varepsilon_\X$. That is,
we do not approximate the projection operator uniformly, but, for a
given $x$, we approximate the projected point adaptively up to a
desired accuracy. This is formalized by computing
points~$\Pr_\X^{\varepsilon}(x)$ with the property that
$\norm{\Pr_\X^{\varepsilon}(x) - \Pr_\X(x)} \leq \varepsilon$ for
every $\varepsilon \geq 0$. This concept of an absolute accuracy of
the projected point is similar in spirit to the adaptive evaluation of
operators as, e.g., used in adaptive wavelet methods (cf. the
$\mathbf{APPLY}$-routine in~\cite{CDD02}). Algorithmically, the idea
is that during the early phases of the algorithm we do not need a
highly accurate projection, and $\Pr_\X^{\varepsilon}(x)$ can be
faster to compute if $\varepsilon$ is larger. In the later phases, one
then adaptively tightens the requirement on the accuracy.

One particularly attractive situation in which the approach works is
the case where~$\X$ is an affine space, i.e., defined by a linear
equation system. Then one can use a truncated iterative method, e.g.,
a conjugate gradient (CG) approach, to obtain an adaptive approximate
projection. We have observed that often only a few steps (2 or 3) of
the CG-procedure are needed to obtain a practically convergent method.

In this paper, we focus on the investigation of convergence properties
of a general variant of the projected subgradient method which relies
on such adaptive approximate projections. We study conditions on the
step sizes and on the accuracy requirements~$\varepsilon_k$ (in each
iteration~$k$) in order to achieve convergence of the sequence of
iterates to an optimal point, or at least convergence of the function
values to the optimum. We investigate two variants of the
algorithm. In the first one, the sequence $(\alpha_k)$ of step sizes
forms a divergent but square-summable series ($\sum \alpha_k =
\infty$, $\sum\alpha_k^2 < \infty$) and is given a priori. The second
variant uses dynamic step sizes which depend on the difference of the
current function value to a constant \emph{target value} that
estimates the optimal value.

A crucial difference of the resulting algorithms to the standard
method is the fact that iterates can be infeasible, i.e., are not
necessarily contained in~$\X$. We thus call the algorithm of this
paper \emph{infeasible-point subgradient algorithm} (ISA). As a
consequence, the objective function values of the iterates might be
smaller than the optimum, which requires a non-standard analysis; see
the proofs in Section~\ref{sect:proofs} for details. Moreover, we
always assume that~$X$ is strictly contained in the interior of the
domain~$\dom f$ of~$f$. Note that this excludes the case $X = \dom f$,
where our algorithm cannot be applied. Furthermore, we assume that
every iterate lies in~$\dom f$, since otherwise no first-order
information is available. This is automatically fulfilled if~$\dom f$
is the whole space, or it can be ensured by requiring that the
accuracies~$\varepsilon_k$ are small enough; cf. also Part~4 of
Remark~\ref{rem:ISA_conv_over}.

This paper is organized as follows. We first discuss related
approaches in the literature. Then we fix some notation and recall a
few basics. In the main part of this paper (Sections \ref{sect:ISA}
and \ref{sect:proofs}), we state our infeasible-point subgradient
algorithm (ISA) and provide proofs of convergence. In the subsequent
sections we briefly discuss some variants of ISA, an example for the
adaptive approximate projection operator from the context of convex
chance constraints, and an application of ISA to the problem of
finding finding the minimum \l-norm solution of an underdetermined
linear equation system, a problem that lately received a lot of
attention in the context of compressed sensing (see,
e.g.,~\cite{D06,CRT06,CSweb}). We finish with some concluding remarks
and give pointers to possible extensions as well as topics of future
research.

\subsection{Related work}

The objective function values of the iterates in subgradient
algorithms typically do not decrease monotonically. With the right
choice of step sizes, the (projected) subgradient method nevertheless
guarantees convergence of the objective function values to the
minimum, see, e.g., \cite{S85,P67,BS81,P78}. A typical result of this
sort holds for step size sequences $(\alpha_k)$ which are nonsummable
($\sum_{k=0}^{\infty}\alpha_k = \infty$), but square-summable
($\sum_{k=0}^{\infty}\alpha_k^2 < \infty$). Thus, $\alpha_k \to 0$ as
$k \to \infty$. Often, the correspon\-ding sequence of points can also
be guaranteed to converge to an optimal solution $x^*$, although this
is not necessarily the case; see~\cite{AW09} for a discussion.

Another widely used step size rule uses an estimate $\varphi$ of the
optimal value~$f^*$, a subgradient $h^k$ of the objective function $f$
at the current iterate $x^k$, and relaxation parameters $\lambda_k >
0$:
\begin{equation}\label{eq:alphak}
  \alpha_k = \lambda_k \frac{f(x^k) - \varphi}{\norm{h^k}_2^2}.
\end{equation}
The parameters $\lambda_k$ are constant or required to obey certain
conditions needed for convergence proofs. The dynamic rule
\eqref{eq:alphak} is a straightforward genera\-lization of the
so-called Polyak-type step size rule, which uses $\varphi = f^*$, to
the more practical case when $f^*$ is unknown. The convergence results
given in~\cite{AHKS87} extend the work of Polyak~\cite{P67,P69} to
$\varphi\geq f^*$ and $\varphi < f^*$ by imposing certain conditions
on the sequence $(\lambda_k)$. We will generalize these results
further, using an adaptive approximate projection operator instead of
the (exact) Euclidean projection.

Many extensions of the basic subgradient scheme exist,
such as variable target value methods (see, e.g.,
\cite{CL02,KAC91,LS05,NB01,SCT00,GK99,BS81}), using approximate
subgradients~\cite{BM73,AIS98,LPS96a,DAF09}, or incremental projection
schemes~\cite{NDP09,NB01,K04}, to name just a few.

Inexact projections have been used previously, probably most
prominently for convex feasibility problems in the framework of
successive projection methods. Indeed, the optimization
problem~(\ref{eq:NLO}) can, at least theoretically, be cast as the
convex feasibility problem to determine $x^* \in X\cap\{f(x)\leq
f^*\}$. Using so-called subgradient projections~\cite{BB96} onto the
second set leads to a subgradient step
\[
x^{k+1} \define x^k - \frac{f(x^k) - f^*}{\|h^k\|^2}h^k,
\]
which corresponds to using a Polyak-type step size without relaxation
parameter, employing the exact optimal value. As illustrated
in~\cite{BB96}, this approach leads to a very flexible framework for
convex feasibility problems as well as (non-smooth) convex
optimization
problems. %; see also the framework proposed in~\cite{NDP09}.

Moreover,~\cite{Z10} considers additive vanishing non-summable error
terms (for both the projection and the subgradient step) and
establishes the existence of a (decaying) bound on the error terms
such that the algorithm will reach a small neighborhood of the optimal
set. However, these bounds are not given explicitly. In contrast, our
results (Theorems~\ref{thm:ISA_conv_apriori} and
\ref{thm:ISA_conv_under}) contain explicit conditions for the error
terms that guarantee convergence to the optimum.
% (see also \cite{SZ98}), Dirk: Referenz erschien mit nicht so
% wesentlich, da sie nur eine Störung im Subgradientenschritt
% behandelt.
% \item the level set subgradient algorithm in \cite{K98}, which employs
%   inexact projections, although here all iterates are strictly
%   feasible; a related article is~\cite{AT09}, where the classical
%   projection is replaced by a non-Euclidean distance-like function.
% \end{itemize}
Another example for the use of inexact projections is the level set
subgradient algorithm in \cite{K98}, although there, all iterates are
strictly feasible.

We emphasize that there are at least three conceptually different
approaches to approximate projections in the present context. The
first concept---prominent, e.g., in the field on convex feasibility
problems---uses the idea of approximating the \emph{direction towards
  the feasible set}, i.e., the iterates approximately move towards the
constraint set. In the second, related, concept one \emph{projects
  exactly onto supersets} of the constraint which are easier to
handle, e.g., half-spaces. With both ideas one can use powerful
notions like Fej{\'e}r-monotonicity or the concept of firmly
non-expansive mappings, see, e.g.,~\cite{BB96} and the more
recent~\cite{LLM09}; see also the ``feasibility operator'' framework
proposed in~\cite{NDP09}. To employ either approach one exploits
analytical knowledge about the feasible set, e.g., that it can be
written as a level set of a known and easy-to-handle convex
function. In the third approach, one aims at \emph{approximating the
  projected point} without further restricting the direction. This
concept applies, for instance, in situations in which a computational
error is made in the projection step (e.g., as in~\cite{Z10}) or when
it is impossible or undesirable to handle the constraints
analytically, but a numerical algorithm is available which calculates
the projection point up to a given accuracy. The adaptive approximate
projections considered in this paper fall under this third category.

Note that, besides the different philosophies and fields of
application, none of the approaches directly dominates the other: On
the one hand, one may move directly towards the feasible set while
missing the projection point, and on the other hand, one may also move
closer to the projected point along a direction which is not towards
the feasible set; see Figure~\ref{fig:approximation_concepts} for an
illustration. However, one can sometimes, for a given rule which
approximates the projection direction, find appropriate half-spaces
which contain the feasible set and realize this approximate projection
exactly. In Section~\ref{sec:examples} we give a concrete example in
which the Fej{\'e}r-type feasibility operator of~\cite{NDP09} is not
applicable, but the exact projection point can be approximated
reasonably well in the sense of our adaptive approximate projection
(see above or~\eqref{eq:IPr} below).

In the present paper we only consider the third approach to
approximate projections and do not use any assumption like
non-expansiveness or Fej{\'e}r-monotonicity for the iteration mapping
in our convergence analyses.

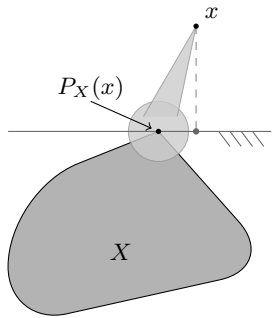
\begin{figure}[tb]
  \centering
  \begin{tikzpicture}
    % Set X
    \draw[fill=black!30] (0,0) [rounded corners=10mm] -- (0,1.8)
    [sharp corners] -- (2,2.6) [rounded corners=8mm] -- (3.6,0.8) -- cycle;
    
    \draw (1.5,1) node{$X$};
    
    % Neighborhood of projected point
    \draw[black!40,fill=black!20,opacity=0.8] (2,2.6) circle (4mm);

    % Cone of projection direction s
    \fill[black!20,opacity=0.8] (2.5,4) -- (1.8,2.8) -- (2.25,2.8) -- cycle;
    \draw[thin,black!40,opacity=0.8] (2.5,4) -- (1.8,2.8);
    \draw[thin,black!40,opacity=0.8] (2.5,4) -- (2.25,2.8);
    
    % Half space
    \draw[black!60] (0,2.6) -- (3.5,2.6);
    \foreach \x in {0,0.15,0.3,0.45}
      \draw[black!60] (2.95+\x,2.4) -- (2.8+\x,2.6);
    
    % Projected onto half space
    \draw[black!60,opacity=0.8,dashed] (2.5,4) -- (2.5,2.6);
    \filldraw[black!60] (2.5,2.6) circle (1pt);

    % Point x
    \fill (2.5,4) circle (1pt) node[above right]{$x$};
    % Projected point P_X(x)
    \fill (2,2.6) circle (1pt);
    \draw (1.1,2.9) node[above]{$P_X(x)$};
    \draw[thin,shorten >= 1mm,->] (1.1,3.0) -- (2,2.6);
  \end{tikzpicture}
  \caption{Schematic illustration of the three concepts of
    ``approximate projections'': The approximation of the projection
    direction (or ``moving towards the feasible set'') moves from $x$
    along a direction within the shaded cone. The exact projection
    onto a half-space containing $X$ moves along the dashed line. The
    approximation of the projected point moves from $x$ into a
    neighborhood of $P_X(x)$, the shaded circle.}
  \label{fig:approximation_concepts}
\end{figure}
% Each of these articles is based on a framework different from ours
% (subgradient bundling, different step size rule or projection operator).
% In particular, it can be seen that the feasibility operator of~\cite{NDP09} is
% not comparable to our projection, i.e., in general, the two concepts do not
% dominate each other. There are cases where the framework from~\cite{NDP09}
% cannot be applied; see also Section~\ref{sec:examples} for a discussion of
% concrete examples.

% The algorithm and convergence results we discuss in this paper
% are---to the best of our knowledge---new. The adaptive projection
% operator we work with %formalizes the most intuitive idea of
% %approximative projections: producing a point ``somewhere in the
% %vicinity'' of the true projection. Thus, it 
% only moves a given point
% to within a certain distance from the corresponding exact projection,
% which is a much less restrictive requirement than imposed on the
% feasibility operator in \cite{NDP09}. Furthermore, we also employ
% dynamic step sizes of the form~\eqref{eq:alphak}, so our results are
% not immediately subsumed in the fairly general framework of~\cite{Z10}
% (although the notion of ``computational errors'' used there can be
% interpreted to cover projection inaccuracies as well).

\subsection{Notation}

In this paper, we consider the convex optimization
problem~\eqref{eq:NLO} in which we assume that $f: \R^n \to
\R\cup\{\infty\}$ is a convex function (not necessarily
differentiable), $\dom f = \{x\in\R^n\suchthat f(x)<\infty\}$, and $\X
\subset\interior(\dom f)\subseteq \R^n$ is a closed convex set (note
that this implies that $f$ is continuous on $\X$).  The set
\begin{equation}\label{eq:sgdef}
  \partial f(x) \define  \{ h \in \R^n \suchthat f(y) \geq f(x)+h^{\top}(y-x)
  \quad \forall\, y \in \R^n\,\}
\end{equation}
is the \emph{subdifferential} of $f$ at a point $x \in \R^n$; its members
are the corresponding \emph{subgradients}. Throughout this paper, we
will assume~\eqref{eq:NLO} to have a nonempty set of optima
\begin{equation}\label{eq:Xstar}
  \X^* \define  \argmin\{f(x) \suchthat x \in \X\}.
\end{equation}
An optimal point will be denoted by $x^*$ and its objective function value
$f(x^*)$ by~$f^*$. For a sequence $(x^k) =(x^0, x^1, x^2, \dots)$ of
points, the corresponding sequence of objective function values will be
abbreviated by $(f_k) = (f(x^k))$.

By $\norm{\cdot}_p$ we denote the usual $\ell_p$-norm, i.e., for $x \in
\R^n$,
\begin{equation}
  \norm{x}_p \define  
  \begin{cases}
    \big(\sum_{i=1}^{n} \abs{x_i}^p\big)^{\frac{1}{p}}, & \text{if }1 \leq p < \infty,\\
    \displaystyle\max_{i=1, \dots, n}\, \abs{x_i}, & \text{if }p = \infty.
  \end{cases}
\end{equation}
If no confusion can arise, we shall simply write $\norm{\cdot}$ instead of
$\norm{\cdot}_2$ for the Euclidean ($\ell_2$-)norm. The Euclidean distance
of a point $x$ to a set $Y$ is
\begin{equation}\label{eq:dist}
  \d_Y(x)\define \inf_{y\in Y} \norm{x-y}_2.
\end{equation}
For $Y$ closed and convex, \eqref{eq:dist} has a unique minimizer, namely
the orthogonal (Euclidean) projection of $x$ onto $Y$, denoted by
$\Pr_Y(x)$.

All further notation will be introduced where it is needed.

\section{The Infeasible-Point Subgradient Algorithm (ISA)}
\label{sect:ISA}

In the projected subgradient algorithm, we replace the exact
projection~$\Pr_\X$ by an adaptive approximate projection. We require
that we can adapt the accuracy of the approximation of the projected point absolutely,
i.e., that for any given accuracy parameter $\varepsilon \geq 0$, the
adaptive approximate projection $\Pr_\X^{\varepsilon} : \R^n \to \R^n$ satisfies
\begin{equation}\label{eq:IPr}
  \norm{\Pr_\X^{\varepsilon}(x) - \Pr_\X(x)} \leq \varepsilon\qquad\text{for all }x\in\R^n.
\end{equation}
In particular, for $\varepsilon = 0$, we have $\Pr_\X^0 =
\Pr_\X$. Note that $\Pr_\X^\varepsilon(x)$ does not necessarily
produce a point that is \emph{closer} to $\Pr_\X(x)$ (or even to $\X$)
than $x$ itself. In fact, this is only guaranteed for $\varepsilon <
\d_X(x)$. 

One example arises in the context of convex chance constraints and is
discussed in Section~\ref{sec:ChanceConstraints}. For the special case
in which~$\X$ is an affine space, we give a detailed discussion of an
adaptive approximate projection satisfying the above requirement in
Section~\ref{sec:CompressedSensing}.

By replacing the exact by an adaptive projection in the projected
subgradient method, we obtain the \emph{Infeasible-point Subgradient
  Algorithm} (ISA), which we will discuss in two variants in the
following.

The stopping criteria of the algorithms will be ignored for the
convergence analyses. In practical implementations, one would stop,
e.g., if no significant progress in the objective (or feasibility) has
occurred within a certain number of iterations.

\subsection{ISA with a predetermined step size sequence}

\begin{algorithm}[t]
  \caption{\textsc{Predetermined Step Size ISA}}
  \label{alg:APrioriISA}
  \begin{algorithmic}[1]
    \REQUIRE{a starting point $x^0$, sequences $(\alpha_k)$, $(\varepsilon_k)$}
    \ENSURE{an (approximate) solution to~\eqref{eq:NLO}}
    \STATE initialize $k \define  0$
    \REPEAT
    \STATE choose a subgradient $h^k \in \partial f(x^k)$ of $f$ at $x^k$
    \STATE compute the next iterate $x^{k+1} \define  \Pr_\X^{\varepsilon_k}\left(x^k - \alpha_k h^k\right)$
    \STATE increment $k \define  k+1$
    \UNTIL{a stopping criterion is satisfied}
  \end{algorithmic}
\end{algorithm}

%\noindent
If the step sizes $(\alpha_k)$ and projection
accuracies~$(\varepsilon_k)$ are \emph{predetermined} (i.e., given a
priori), we obtain Algorithm~\ref{alg:APrioriISA}. Note that $h^k = 0$
might occur, but does not necessarily imply that~$x^k$ is optimal,
because $x^k$ may be infeasible. In such a case, the adaptive projection will
change $x^k$ to a different point as soon as $\varepsilon_k$ becomes
small enough.

We will now state our main convergence result for this variant of the
ISA, using fairly standard step size conditions. The proof is provided
in Section~\ref{sect:proofs}.

\begin{theorem}[Convergence for predetermined step size sequences]\label{thm:ISA_conv_apriori}\\
  Let the projection accuracy sequence $(\varepsilon_k)$ be such that
  \begin{equation}\label{eq:eps}
    \varepsilon_k \geq 0,\quad \sum_{k=0}^\infty \varepsilon_k < \infty,
  \end{equation}
  let the positive step size sequence $(\alpha_k)$ be such that
  \begin{equation}\label{eq:alpha}
    \sum_{k=0}^\infty \alpha_k = \infty,\quad\sum_{k=0}^\infty \alpha_k^2 <
    \infty,
  \end{equation}
  and let the following relation hold:
  \begin{equation}\label{eq:alphageqepsrest}
    \alpha_k \geq \sum_{j=k}^\infty \varepsilon_j\qquad\forall\, k=0,1,2,\dots
  \end{equation}
  Suppose $\norm{h^k} \leq H < \infty$ for all $k$. Then the sequence
  of the ISA iterates $(x^k)$ converges to an optimal point.
\end{theorem}

\begin{remark}
  Relations~\eqref{eq:eps},~\eqref{eq:alpha}, and
  \eqref{eq:alphageqepsrest} can be ensured, e.g., by the sequences
  $\varepsilon_k = 1/k^2$ and $\alpha_k = 1/(k-1)$ for $k>1$; in
  particular,
  \[
  \sum_{j=k}^\infty \varepsilon_k \leq \int_{k-1}^\infty
  \frac{1}{x^2}\,dx = \frac{1}{k-1} = \alpha_k.
  \]
\end{remark}

\subsection{ISA with dynamic step sizes}

%\noindent 
In order to apply the dynamic step size rule~\eqref{eq:alphak}, we
need several modifications of the basic method, yielding
Algorithm~\ref{alg:DynamicISA}. This algorithm works with an
estimate~$\varphi$ of the optimal objective function value $f^*$ and
essentially tries to reach a feasible point~$x^k$ with $f(x^k)
\leq\varphi$. (Note that if $\varphi = f^*$, we would have obtained an
optimal point in this case.) %To eventually become feasible, we will again need $\varepsilon_k\to 0$.

\begin{algorithm}[t]
  \caption{\textsc{Dynamic Step Size ISA}}
  \label{alg:DynamicISA}
  % \begin{algorithmic}[1]
  %   \REQUIRE{estimate $\varphi$ of $f^*$, starting point $x^0$
  %     with $f_0 = f(x^0) \geq \varphi$, sequences $(\lambda_k)$,
  %     $(\varepsilon_k)$} 
  %   \ENSURE{an (approximate) solution to \eqref{eq:NLO}}
  %   \STATE initialize $k \define 0$
  %   \REPEAT 
  %   \STATE set $f_k \define f(x^k)$
  %   \STATE choose a subgradient $h^k \in \partial f(x^k)$ of $f$ at
  %   $x^k$\label{step:subgradient}
  %   \IF{$h^k = 0$}\label{step:SubgradZero}
  %   \IF{$x^k \in \X$} 
  %   \STATE stop (at optimal feasible point $x^k \in
  %   \X^*$)\label{step:terminate_opt}
  %   \ELSE 
  %   \STATE compute the next iterate $x^{k+1} \define \Pr_\X^{0}(x^k)$\label{step:ZeroSubgradProj}
  %   \ENDIF
  %   \ELSE
  %   \STATE compute step size $\alpha_k \define \lambda_k(f(x^k) - \varphi)/\norm{h^k}^2$
  %   \STATE compute the next iterate $x^{k+1} \define
  %   \Pr_\X^{\varepsilon_k}(x^k - \alpha_k h^k)$\label{step:NextIterate}
  %   \IF{$f(x^{k+1}) \leq \varphi$ and $\varepsilon_k > 0$}\label{step:InfBelowTarget}
  %   \STATE set $x^{k+1} \define  \Pr_\X^{0}(x^k - \alpha_k h^k)$\label{step:ExactProjection}
  %   \ENDIF
  %   \ENDIF
  %   \IF{$f(x^{k+1}) \leq \varphi$}\label{step:BelowTarget}
  %   \STATE stop (at feasible point $x^{k+1} \in \X$ with $f^*\leq f(x^{k+1}) \leq \varphi$)\label{step:terminate}
  %   \ENDIF
  %   \STATE increment $k \define  k+1$
  %   \UNTIL{a stopping criterion is satisfied}
  % \end{algorithmic}
  \begin{algorithmic}[1]
    \REQUIRE{estimate $\varphi$ of $f^*$, starting point $x^0$, sequences $(\lambda_k)$,
      $(\varepsilon_k)$}, parameter $\gamma\in(0,1)$ 
    \ENSURE{an (approximate) solution to \eqref{eq:NLO}}
    \STATE initialize $k \define 0$, $\ell=-1$, $x^{-1}\define x^0$, $h^{-1}\define 0$, $\alpha_{-1}\define 0$, $\varepsilon_{-1}\define \varepsilon_0$
    \REPEAT 
    \STATE choose a subgradient $h^k \in \partial f(x^k)$ of $f$ at $x^k$\label{step:subgradient}
    \IF{$f_k\leq\varphi\textbf{ or }h^k = 0$}\label{step:SubgradZero}
    %%%\IF{$f_k\leq\varphi$}
    \IF{$x^k \in \X$} \label{step:feasible}
    \STATE stop (at feasible point $x^k$ showing $\varphi\geq f^*$; optimal if $h^k=0$)\label{step:terminate_opt}
    %\ELSE 
    \ENDIF
    \STATE increment $\ell\define\ell+1$, reset $x^{k} \define \Pr_\X^{\varepsilon}(x^{k-1} - \alpha_{k-1} h^{k-1})$ for $\varepsilon = \gamma^\ell\varepsilon_{k-1}$\label{step:refine_proj}
    \STATE go to Step \ref{step:subgradient}
    %\ELSE
    \ENDIF    
%    \STATE choose a subgradient $h^k \in \partial f(x^k)$ of $f$ at $x^k$
%    \IF{$h^k=0$}
%    \STATE go to step \ref{step:feasible}
%    \ENDIF    
    \STATE compute step size $\alpha_k \define \lambda_k(f_k - \varphi)/\norm{h^k}^2$
    \STATE \hspace*{-0.5pt}compute the next iterate $x^{k+1} \define \Pr_\X^{\varepsilon_{k}}(x^k - \alpha_k h^k)$\label{step:NextIterate}
    \STATE \hspace*{-0.5pt}reset $\ell\define 0$ and increment $k \define  k+1$
    %\ENDIF
    \UNTIL{a stopping criterion is satisfied}
  \end{algorithmic}
\end{algorithm}

\begin{remark} A few comments on Algorithm~\ref{alg:DynamicISA} are in order: 
\begin{enumerate}
\item Since $0<\gamma<1$, $\gamma^\ell \to 0$ (strictly monotonically)
  for $\ell\to\infty$. Thus,
  Steps~\ref{step:subgradient}--\ref{step:refine_proj} constitute a
  \emph{projection accuracy refinement phase}, i.e., an inner loop in
  which the current $k$ is temporarily fixed, and $x^k$ is
  \emph{recomputed} with a stricter accuracy setting for the adaptive
  projection. This phase either leads to a point showing $\varphi\geq
  f^*$ (by termination or convergence in the inner loop over $\ell$)
  or eventually resets $x^k$ to a point with $f_k>\varphi$ and
  $h^k\neq 0$ so that the regular (outer) iteration is resumed (with
  $k$ no longer fixed).
\item Note that, if $x^0$ is such that $f_0\leq \varphi$ or $h^k=0$,
  the algorithm begins with such a refinement phase, projecting $x^0$
  more and more accurately until neither case holds any longer (if
  possible); the initializations with counter $-1$ are needed for this
  eventuality.
  % Clearly, $\varepsilon_{-1}$ needs not be initialized
  % to $\varepsilon_0$, but any nonnegative constant would
  % suffice. 
  Moreover, we could clearly postpone the (repeated)
  determination of a subgradient (Step~\ref{step:subgradient}) in a
  refinement phase until $f_k>\varphi$ is achieved, i.e., $h^k=0$
  would be the only reason for another accuracy refinement. This may
  be important in practice, where finding a subgradient sometimes is
  expensive itself, and the case $h^k=0$ presumably occurs very rarely
  anyway. For the sake of brevity we did not treat this explicitly in
  Algorithm~\ref{alg:DynamicISA}.
\item There are various ways in which the accuracy refinement phase
  could be realized. Instead of $(\gamma^\ell)$ with constant
  $\gamma\in(0,1)$, any (strictly) monotonically decreasing sequence
  $(\gamma_\ell)$ could be used. Since we will need $\varepsilon_k\to
  0$ to achieve feasibility (in the limit) anyway, which implies that
  for all $k$ there always exists some $L>0$ such that
  $\varepsilon_{k+L}<\varepsilon_k$, we could also use
  $\min\{\varepsilon_{k-1},\varepsilon_{k-1+\ell}\}$ as the
  recalibrated accuracy. Moreover, we do not need to fix~$k$, i.e.,
  repeatedly \emph{replace} $x^k$ by finer approximate projections,
  but could produce a finite series of identical iterates (each reset
  to the last one before the inner loop started) until the refinement
  phase is over. %, or set $\alpha_k\define 0$ in the problematic cases
  %($f_k\leq\varphi$ or $h^k=0$) and thus not even fix the point to be
  %projected but instead project ``problematic'' iterates without
  %taking a subgradient step first. Note that this last variant
  %corresponds to using step sizes
  Similarly, we could use
  $\alpha_k=\max\{0,\lambda_k(f_k-\varphi)/\norm{h^k}^2\}$ (and $0$ if
  $h^k=0$); letting $\varepsilon_k\to 0$ then naturally implements the
  refinement, while in iterations with $\alpha_k=0$, the produced
  point may move up to $\varepsilon_k$ \emph{away} from the optimal
  set. Assuming $(\varepsilon_k)$ is summable, this does not impede
  convergence. For all these variants, analogues to the following
  convergence results hold true as well; however, the proofs require
  some extensions to account for the technical differences to the
  variant we chose to present, which admitted the overall shortest
  proofs. In practice, we would generally expect these variants to
  behave similarly. Furthermore, note that in principle, the
  ``problematic'' cases could also be treated by reverting to exact
  projections; however, in our present context this should be avoided
  since computing the exact projection is considered too expensive.

% \item If during the algorithm we obtain an infeasible point~$x^{k+1}$
%   with $f(x^{k+1}) \leq \varphi$, the next step size would be zero or
%   negative, see~\eqref{eq:alphak}. In this case, we perform an
%   \emph{exact} projection in Step~\ref{step:ExactProjection} (note
%   that this step can be replaced by an adaptive approximate projection
%   with decreasing~$\varepsilon$ until we reach $f(x^{k+1}) > \varphi$
%   or $\varepsilon = 0$). If the new point~$x^{k+1} \in \X$ still
%   satisfies~$f(x^{k+1}) \leq \varphi$, we terminate
%   (Step~\ref{step:terminate}) with a feasible point showing
%   that~$\varphi$ is too large. In this case, one can
%   decrease~$\varphi$ and iterate, thus resorting to a kind of a
%   variable target value method; see Section~\ref{subsect:VTVM} for a
%   more detailed discussion.
% \item If $h^k = 0$ occurs during the algorithm, the step
%   size~\eqref{eq:alphak} is meaningless. If in this case~$x^k$ is
%   feasible, it must be optimal, i.e., we have reached an unconstrained
%   optimum that lies within~$\X$. Otherwise, we perform an exact
%   projection in Step~\ref{step:ZeroSubgradProj} (or iteratively
%   decrease~$\varepsilon$ as mentioned above). The new point~$x^{k+1}$
%   will either yield~$h^{k+1} \neq 0$ or an unconstrained optimum.
\end{enumerate}
\end{remark}  
We obtain the following convergence results, depending on
whether~$\varphi$ over- or underestimates~$f^*$. The proofs are
deferred to the next section.

\begin{theorem}[Convergence for dynamic step sizes with overestimation]\label{thm:ISA_conv_over}
  Let the optimal point set~$X^*$ be bounded, $\varphi \geq f^*$, $0 <
  \lambda_k \leq \beta < 2$ for all $k$, and $\sum_{k=0}^{\infty}
  \lambda_k = \infty$. Let $(\nu_k)$ be a nonnegative sequence with
  $\sum_{k=0}^{\infty} \nu_k < \infty$, and let
  \begin{align}
    \nonumber \overline{\varepsilon}_k \define  &
    - \left(\frac{\lambda_k(f_k - \varphi)}{\norm{h^k}} + \d_{X^*}(x^k)\right)\\
    & + \sqrt{\left(\frac{\lambda_k(f_k - \varphi)}{\norm{h^k}} + 
        \d_{X^*}(x^k)\right)^2 + \frac{\lambda_k(2 - \lambda_k)(f_k - \varphi)^2}{\norm{h^k}^2}}.\label{eq:epsISAover}
  \end{align}
  If the subgradients~$h^k$ satisfy $0 < \underline{H} \leq \norm{h^k} \leq
  \overline{H} < \infty$ and $(\varepsilon_k)$ satisfies $0 \leq \varepsilon_k
  \leq \min\{\overline{\varepsilon}_k,\,\nu_k\}$ for all $k$, then the
  following holds.
  \begin{enumerate}
  \item[\emph{~(i)}] For any given $\delta > 0$ there exists some
    index $K$ such that $f(x^K) \leq \varphi + \delta$.
  \item[\emph{(ii)}] If additionally $f(x^k) > \varphi$ for all $k$
    and if $\lambda_k \to 0$, then $f_k\to\varphi$ for $k\to\infty$.
  \end{enumerate}
\end{theorem}

\begin{remark}\label{rem:ISA_conv_over}\
  \begin{enumerate}
  \item The sequence $(\nu_k)$ is a technicality needed in the proof
    to ensure $\varepsilon_k \to 0$. Note from~\eqref{eq:epsISAover}
    that $\overline{\varepsilon}_k > 0$ as long as ISA keeps
    iterating (in the main loop over $k$), since $f_k > \varphi$ is 
    then guaranteed by the adaptive accuracy refinements and $0 < \lambda_k
    < 2$ holds by assumption.
  \item More precisely, part (i) of Theorem~\ref{thm:ISA_conv_over} essentially means
    that after a finite number of iterations, we reach a point~$x^k$
    with $f^*-c \leq f(x^k) \leq \varphi + \delta$ for any $c>0$. If $\varphi < f(x^k) \leq \varphi
    + \delta$, this point may still be infeasible, but the closer 
    $f(x^k)$ gets to~$\varphi$, the smaller
    $\overline{\varepsilon}_k$ becomes, i.e., the algorithm automatically %adaptively
    increases the projection accuracy.
    % (Thus, one can expect the
    % possible feasibility violation to be reasonably small, depending
    % on the quality of the estimate $\varphi$ and the value of the
    % constant $\delta$.) 
    On the other hand, termination in Step~\ref{step:terminate_opt}
    implies that $f(x^k)\geq f^*$ (since $x^k$ is then feasible), and
    if some inner loop is infinite, then the refined projection points
    converge to a feasible point. Hence, for every $c>0$, there is
    some integer $0\leq L<\infty$ such that after the $L$-th accuracy
    refinement and replacement of $x^k$, $f(x^k)\geq f^*-c$.
  \item Part (ii) shows what happens when all function values~$f(x^k)$
    stay above the overestimate~$\varphi$ of~$f^*$---which
    particularly holds true \emph{after} possible refinements, if all
    the accuracy refinement phases are finite (and no termination
    occurs)---and we impose $\lambda_k \to 0$ for $k \to \infty$: We
    eventually obtain~$f(x^k)$ arbitrarily close to~$\varphi$, with
    vanishing feasibility violation as $k \to \infty$. Then, as well
    as in case of termination in Step~\ref{step:terminate_opt} or
    convergence in a refinement phase ($\ell\to\infty$), it may be
    desirable to restart the algorithm using a smaller~$\varphi$; see
    Section~\ref{subsect:VTVM}.
  \item The conditions $\norm{h^k} \geq \underline{H} > 0$, for all~$k$, in
    Theorem~\ref{thm:ISA_conv_over} imply that all subgradients used
    by the algorithm are nonzero. %In this case, Steps~\ref{step:SubgradZero}--\ref{step:ZeroSubgradProj} are never executed. 
    These conditions are often automatically guaranteed, for
    example, if~$X$ is compact and no unconstrained optimum of~$f$
    lies in~$X$. In this case, $\norm{h} \geq \underline{H} > 0$ for all $h
    \in \partial f(x)$ and $x \in X$. Moreover, the same holds for a
    small enough open neighborhood of~$X$. Also, the norms of the
    subgradients are bounded from above. Thus, if we start close
    enough to~$X$ and restrict~$\varepsilon_k$ to be small enough, the
    conditions of Theorem~\ref{thm:ISA_conv_over} are
    fulfilled. Another example in which the conditions are satisfied
    appears in Section~\ref{sec:CompressedSensing}.
  \end{enumerate}
\end{remark}

\begin{theorem}[Convergence for dynamic step sizes with underestimation]\label{thm:ISA_conv_under}
  Let the set of optimal points~$X^*$ be bounded, $\varphi < f^*$, $0
  < \lambda_k \leq \beta < 2$ for all~$k$, and $\sum_{k=0}^{\infty}
  \lambda_k = \infty$. Let $(\nu_k)$ be a nonnegative sequence with
  $\sum_{k=0}^{\infty} \nu_k < \infty$, let
  \begin{equation}\label{eq:LkDef}
    L_k \define  \frac{\lambda_k(2 - \beta)(f_k -
      \varphi)}{\norm{h^k}^2}\big(f^* - f_k + \frac{\beta}{2 - \beta}(f^* - \varphi)\big),
  \end{equation}
  and let
  \begin{equation}\label{eq:epsISAunder}
    \tilde{\varepsilon}_k \define  - \left(\frac{\lambda_k(f_k - \varphi)}{\norm{h^k}}
      + \d_{X^*}(x^k)\right) + \sqrt{\left(\frac{\lambda_k(f_k - \varphi)}{\norm{h^k}}
        +\d_{X^*}(x^k)\right)^2 - L_k}.
  \end{equation}
  If the subgradients~$h^k$ satisfy $0 < \underline{H} \leq \norm{h^k} \leq
  \overline{H} < \infty$ and $(\varepsilon_k)$ satisfies $0 \leq \varepsilon_k
  \leq \min\{\abs{\tilde{\varepsilon}_k},\,\nu_k\}$ for all $k$, then
  the following holds.
  \begin{enumerate}
  \item[\emph{~(i)}] For any given $\delta > 0$, there exists some $K$
    such that $f_K \leq f^* +
    \frac{\beta}{2-\beta}(f^*-\varphi)+\delta$.
  \item[\emph{(ii)}] If additionally $\lambda_k \to 0$, then the
    sequence of objective function values $(f_k)$ of the ISA iterates
    $(x^k)$ converges to the optimal value $f^*$.
  \end{enumerate}
\end{theorem}

\begin{remark}\label{rem:ISA_conv_under}\
  \begin{enumerate}
  \item If $f(x^k) \leq \varphi < f^*$,
    Steps~\ref{step:subgradient}--\ref{step:refine_proj} ensure that
    after a finite number of projection refinements~$x^k$ satisfies
    $\varphi < f(x^{k})$. Thus, the algorithm will never terminate
    with Step~\ref{step:terminate_opt} and every refinement phase is
    finite.
  \item Moreover, infeasible points~$x^k$ with $\varphi < f(x^k) <
    f^*$ are possible. Hence, the inequality in
    Theorem~\ref{thm:ISA_conv_under}~(i) may be satisfied too soon to
    provide conclusive information regarding solution quality.
    Interestingly, part (ii) shows that by letting the parameters
    $(\lambda_k)$ tend to zero, one can nevertheless establish
    convergence to the optimal value $f^*$ (and $d_{\X}(x^k)\leq
    d_{{\X}^{*}}(x^k)\to 0$, i.e., asymptotic feasibility).
  \item Theoretically, small values of~$\beta$ yield smaller errors,
    while in practice this restricts the method to very small steps
    (since $\lambda_k \leq \beta$), resulting in slow
    convergence. This illustrates a typical kind of trade-off between
    solution accuracy and speed.
  \item The use of $\abs{\tilde{\varepsilon}_k}$ in
    Theorem~\ref{thm:ISA_conv_under} avoids conflicting bounds on
    $\varepsilon_k$ in case $L_k > 0$. Because $0 \leq \varepsilon_k
    \leq \nu_k$ holds notwithstanding, $0\leq \varepsilon_k\to 0$ is
    maintained.
  \item The same statements on lower and upper bounds on~$\norm{h^k}$
    as in Remark~\ref{rem:ISA_conv_over} apply in the context of
    Theorem~\ref{thm:ISA_conv_under}.
  \end{enumerate}
\end{remark}

\section{Convergence of ISA}\label{sect:proofs}

From now on, let $(x^k)$ denote the sequence of points with
corresponding objective function values $(f_k)$ and subgradients
$(h^k)$, $h^k\in\partial f(x^k)$, as generated by ISA in the
respective variant under consideration.

Let us consider some basic inequalities which will be essential in
establishing our main results. The exact Euclidean projection is
nonexpansive, therefore
\begin{equation}\label{eq:distPrx}
  \lVert \Pr_\X(y)-x\rVert\leq \norm{y-x} \quad\forall x\in\X.
\end{equation}
Hence, for the adaptive approximate projection $\Pr_\X^{\varepsilon}$ we have, by
(\ref{eq:IPr}) and (\ref{eq:distPrx}), for all $x\in\X$
\begin{align}
  \nonumber\lVert\Pr_\X^{\varepsilon}(y)-x\rVert&=\lVert\Pr_\X^{\varepsilon}(y)-\Pr_\X(y)+\Pr_\X(y)-x\rVert\\
  \label{eq:distIPrx}&\leq\lVert\Pr_\X^{\varepsilon}(y)-\Pr_\X(y)\rVert+\lVert\Pr_\X(y)-x\rVert\leq\varepsilon +\lVert y-x\rVert.
\end{align}
At some iteration $k$, let $x^{k+1}$ be produced by ISA using some
step size~$\alpha_k$ and write $y^{k}\define x^k-\alpha_k h^k$. We
thus obtain for every $x\in\X$: 
\begingroup 
\allowdisplaybreaks
\begin{align}
  \nonumber &\lVert x^{k+1}-x\rVert^2=\lVert\Pr_\X^{\varepsilon_k}(y^{k})-x\rVert^2\\
  \nonumber \leq~&\left(\lVert y^{k}-x\rVert+\varepsilon_k\right)^2=\lVert y^{k}-x\rVert^2+2\,\lVert y^{k}-x\rVert\,\varepsilon_k+\varepsilon_k^2\\
  \nonumber =~&\lVert x^k-x\rVert^2-2\,\alpha_k(h^k)^\top(x^k-x)+\alpha_k^2\,\norm{h^k}^2+2\,\lVert y^{k}-x\rVert\, \varepsilon_k+\varepsilon_k^2\\
  \nonumber \leq~&\lVert x^k-x\rVert^2-2\,\alpha_k(f_k-f(x))+\alpha_k^2\,\norm{h^k}^2+2\lVert x^k-x\rVert\varepsilon_k+2\,\alpha_k\,\varepsilon_k\norm{h^k}+\varepsilon_k^2\\
  \label{eq:distIPrxSq}=~&\lVert x^k-x\rVert^2-2\,\alpha_k(f_k-f(x))+\left(\alpha_k\,\norm{h^k} +\varepsilon_k\right)^2+2\,\lVert x^k-x\rVert\,\varepsilon_k,
\end{align}
\endgroup where the second inequality follows from the subgradient
definition \eqref{eq:sgdef} and the triangle inequality. Note that the
above inequalities \eqref{eq:distPrx}--\eqref{eq:distIPrxSq} hold in
particular for every optimal point $x^*\in\X^*$.

\subsection{ISA with predetermined step size sequence}\label{sect:ISA_conv_apriori}

The proof of the convergence of the ISA iterates~$x^k$ is somewhat
more involved than for the classical subgradient method as, e.g.,
in~\cite{S85}. This is due to the additional error terms by adaptive approximate
projection and the fact that $f_k\geq f^*$ is not guaranteed since the
iterates may be infeasible.
\medskip

\noindent
\textbf{Proof of Theorem~\ref{thm:ISA_conv_apriori}.}~ We rewrite the
estimate \eqref{eq:distIPrxSq} with $x=x^*\in X^*$ as
  \begin{equation}
    \label{eq:basic_isa_divergent_series}
    \lVert x^{k+1}-x^*\rVert^2 \leq \lVert
    x^k-x^*\rVert^2-2\,\alpha_k\,(f_k-f^*)+\underbrace{\left(\alpha_k\lVert
        h^k\rVert +\varepsilon_k\right)^2+2\,\lVert
      x^k-x^*\rVert\,\varepsilon_k}_{\eqqcolon\beta_k}
  \end{equation}
  and obtain (by applying \eqref{eq:basic_isa_divergent_series} for
  $k=0,\dots,m$)
  \begin{align}
    \nonumber \lVert x^{m+1}-x^*\rVert^2~\leq~&\lVert
    x^0-x^*\rVert^2-2\sum_{k=0}^m (f_k-f^*)\alpha_k+\sum_{k=0}^m
    \beta_k.
  \end{align}
  Our first goal is to show that $\sum_k\beta_k$ is a convergent
  series. Using $\norm{h^k}\leq H$ and denoting $A\define
  \sum_{k=0}^\infty \alpha_k^2$, we get
  \begin{equation}
    \label{eq:3}
    \nonumber \sum_{k=0}^m\beta_k~\leq~
    AH^2+\sum_{k=0}^m \varepsilon_k^2+2H\sum_{k=0}^m \alpha_k \varepsilon_k +2\sum_{k=0}^m \lVert x^k-x^*\rVert\varepsilon_k.
  \end{equation}
  Now denote $D \define \lVert x^0-x^*\rVert$ and consider the last
  term (without the factor $2$): 
  \begingroup \allowdisplaybreaks
  \begin{align}
    \nonumber &\sum_{k=0}^m \lVert x^k-x^*\rVert\varepsilon_k~=~D\,\varepsilon_0+\sum_{k=1}^m\big\lVert\Pr_\X^{\varepsilon_{k-1}}\left(x^{k-1}-\alpha_{k-1}h^{k-1}\right)-x^*\big\rVert\,\varepsilon_k\\
    \nonumber \leq~&D\,\varepsilon_0+\sum_{k=1}^m\big\lVert\Pr_\X^{\varepsilon_{k-1}}\left(x^{k-1}-\alpha_{k-1}h^{k-1}\right)-\Pr_\X\left(x^{k-1}-\alpha_{k-1}h^{k-1}\right)\big\rVert\,\varepsilon_k\\
    \nonumber &\qquad\hspace*{-2pt}+\sum_{k=1}^m\big\lVert\Pr_\X\left(x^{k-1}-\alpha_{k-1}h^{k-1}\right)-x^*\big\rVert\,\varepsilon_k\\
    \nonumber \leq~&D\,\varepsilon_0+\sum_{k=1}^m\varepsilon_{k-1}\varepsilon_k +\sum_{k=1}^m\big\lVert x^{k-1}-\alpha_{k-1}h^{k-1}-x^*\big\rVert\,\varepsilon_k\\
    \nonumber \leq~&D\,\varepsilon_0+\sum_{k=0}^{m-1}\varepsilon_k\varepsilon_{k+1}+\sum_{k=0}^{m-1}\lVert x^k-x^*\rVert\,\varepsilon_{k+1}+\sum_{k=0}^{m-1}\norm{h^k}\,\alpha_k\,\varepsilon_{k+1}\\
    \label{eq:distweg1} \leq~&D\,(\varepsilon_0+\varepsilon_1)+\sum_{k=0}^{m-1}\varepsilon_k\varepsilon_{k+1}+\sum_{k=1}^{m-1}\lVert x^k-x^*\rVert\,\varepsilon_{k+1}+H\sum_{k=0}^{m-1}\alpha_k\,\varepsilon_{k+1}.
  \end{align}
  \endgroup Repeating this procedure to eliminate all terms $\lVert
  x^k-x^*\rVert$ for $k>0$, we obtain 
  \begingroup 
  \allowdisplaybreaks
  \begin{align}
    \nonumber \text{(\ref{eq:distweg1})}~\leq~\dots~\leq~&D\sum_{k=0}^m \varepsilon_k+\sum_{j=1}^m\Big(\sum_{k=0}^{m-j}\varepsilon_k \varepsilon_{k+j}+H\sum_{k=0}^{m-j}\alpha_k\varepsilon_{k+j}\Big)\\
    \label{eq:distweg2} =~&D\sum_{k=0}^m\varepsilon_k+\sum_{j=1}^{m}\sum_{k=0}^{m-j}(\varepsilon_k+H\alpha_k)\,\varepsilon_{k+j}.
  \end{align}
  \endgroup Using the above chain of inequalities, \eqref{eq:eps} and
  \eqref{eq:alphageqepsrest}, and the abbreviation $E\define
  \sum_{k=0}^\infty \varepsilon_k$, we finally get: 
  \begingroup
  \allowdisplaybreaks
  \begin{align}
    \nonumber &\lVert x^{m+1}-x^*\rVert^2 + 2\sum_{k=0}^m (f_k-f^*)\,\alpha_k \leq D^2 + \sum_{k=0}^m \beta_k\\
    \nonumber\leq~&D^2+AH^2+\sum_{k=0}^m\varepsilon_k^2+2\,H\sum_{k=0}^m\alpha_k\varepsilon_k+2\,D\sum_{k=0}^m \varepsilon_k+2\sum_{j=1}^{m}\sum_{k=0}^{m-j}(\varepsilon_k+H\alpha_k)\,\varepsilon_{k+j}\\
    \nonumber \leq~&D^2+AH^2+2\,D\sum_{k=0}^m\varepsilon_k+2\sum_{j=0}^m \sum_{k=0}^{m-j}\varepsilon_k\varepsilon_{k+j}+2\,H\sum_{j=0}^m\sum_{k=0}^{m-j}\alpha_k\varepsilon_{k+j}\\
    \nonumber =~&D^2+AH^2+2\,D\sum_{k=0}^m\varepsilon_k+2\sum_{j=0}^m\Big(\varepsilon_j\sum_{k=j}^m\varepsilon_k\Big)+2\,H\sum_{j=0}^m\Big(\alpha_j\sum_{k=j}^m\varepsilon_k\Big)\\
    \nonumber \leq~&D^2+AH^2+2\,D\sum_{k=0}^m\varepsilon_k+2\sum_{j=0}^m E\, \varepsilon_j+2H\sum_{j=0}^m \alpha_j\, \alpha_j\\
    \nonumber \leq~&D^2+AH^2+2\,(D+E)\sum_{k=0}^m\varepsilon_k+2\,H\sum_{k=0}^m\alpha_k^2\\
    \label{eq:5} \leq~&(D+E)^2+E^2+(2+H)\,A\,H~=:~R~<~\infty.
  \end{align}
  \endgroup
  
  Since the iterates $x^k$ may be infeasible, possibly $f_k<f^*$, and
  hence the second term on the left hand side of~\eqref{eq:5} might
  be negative. Therefore, we distinguish two cases:
  
  \begin{enumerate}
  \item[i)] If $f_k\geq f^*$ for all but finitely many $k$, we can
    assume without loss of generality that $f_k\geq f^*$ for all $k$
    (by considering only the ``later'' iterates). Now, because
    $f_k\geq f^*$ for all $k$,
    \[
    \sum_{k=0}^{m}(f_k-f^*)\,\alpha_k \geq
    \sum_{k=0}^{m}\Big(\underbrace{\min_{j=0,\dots,m} f_j}_{\eqqcolon f^*_m}-f^*\Big)\,\alpha_k = (f^*_m-f^*)\sum_{k=0}^{m}\alpha_k.
    \]
    Together with (\ref{eq:5}) this yields
    \[
    0\leq 2\,(f^*_m-f^*)\sum_{k=0}^{m}\alpha_k\leq R\quad\Longleftrightarrow\quad 0\leq f^*_m-f^*\leq\frac{R}{2\sum_{k=0}^{m}\alpha_k} .
    \]
    Thus, because $\sum_{k=0}^{m}\alpha_k$ diverges, we have $f^*_m\to
    f^*$ for $m\to \infty$ (and, in particular,
    $\liminf_{k\to\infty}f_k = f^*$).

    \hspace*{1.5em}To show that $f^*$ is in fact the only possible
    accumulation point (and hence the limit) of $(f_k)$, assume that
    $(f_k)$ has another accumulation point strictly larger than $f^*$,
    say $f^*+\eta$ for some $\eta>0$.  Then, both cases $f_k<
    f^*+\tfrac{1}{3}\eta$ and $f_k> f^*+\tfrac{2}{3}\eta$ must occur
    infinitely often. We can therefore define two index subsequences
    $(m_{\ell})$ and $(n_{\ell})$ by setting $n_{(-1)}\define -1$ and,
    for $\ell\geq 0$,
    \begin{align*}
      m_{\ell} &\define \min\{\, k \suchthat k>n_{\ell-1},~f_k> f^*+\tfrac{2}{3}\eta\,\},\\
      n_{\ell} &\define \min\{\, k \suchthat k>m_{\ell},~f_k< f^*+\tfrac{1}{3}\eta\,\}.
    \end{align*}
    Figure \ref{fig:updown} illustrates this choice of indices.
    \begin{figure}[t]
      \centering
      \includegraphics[width=0.9\textwidth]{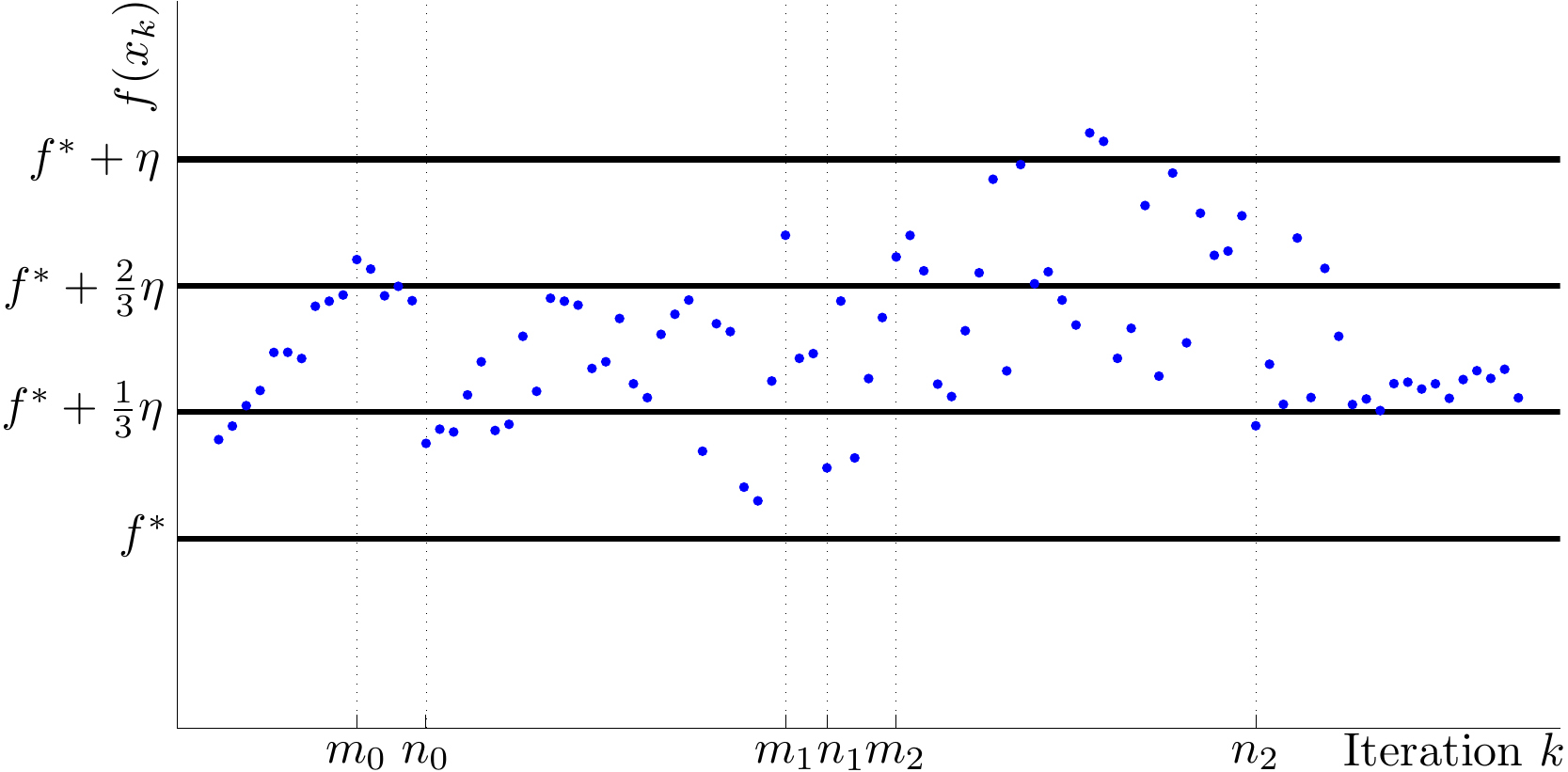}
      \caption{The sequences $(m_\ell)$ and $(n_\ell)$.}
      \label{fig:updown}
    \end{figure}     
    Now observe that for any $\ell$, 
    \begin{align}
      \nonumber \tfrac{1}{3}\eta &< f_{m_{\ell}}-f_{n_{\ell}} \leq H\cdot\norm{x^{n_{\ell}}-x^{m_{\ell}}} 
      \leq H\left(\norm{x^{n_{\ell}-1}-x^{m_{\ell}}}+H\alpha_{n_{\ell}-1}+\varepsilon_{n_{\ell}-1}\right)\\
      \label{eq:est_fdiff_eta_i} &\leq \dots \leq H^2 \sum_{j=m_{\ell}}^{n_{\ell}-1}\alpha_j + H\sum_{j=m_{\ell}}^{n_{\ell}-1}\varepsilon_j,
    \end{align}
    where the second inequality is obtained similar to~\eqref{eq:distweg1}.
    For a given $m$, let $\ell_{m}\define\max\{\,\ell \suchthat
    n_{\ell}-1\leq m\,\}$ be the number of blocks of indices between
    two consecutive indices $m_\ell$ and $n_\ell-1$ until~$m$. We
    obtain:
    \begin{equation}\label{eq:est_fdiff_eta_i_sum}
    \tfrac{1}{3}\sum_{\ell=0}^{\ell_m}\eta\leq H^2\sum_{\ell=0}^{\ell_m}\sum_{j=m_{\ell}}^{n_{\ell}-1}\alpha_j +H\sum_{\ell=0}^{\ell_m}\sum_{j=m_{\ell}}^{n_{\ell}-1}\varepsilon_j
    \leq H^2\sum_{\ell=0}^{\ell_m}\sum_{j=m_{\ell}}^{n_{\ell}-1}\alpha_j +HE.
    \end{equation}
    For $m\to\infty$, the left hand side tends to infinity, and since
    $HE < \infty$, this implies that
    \[
    \sum_{\ell=0}^{\ell_m}\sum_{j=m_{\ell}}^{n_{\ell}-1}\alpha_j\to\infty.
    \]
    Then, since $\alpha_k>0$ and $f_k\geq f^*$ for all $k$, (\ref{eq:5})
    yields
    \begin{align*}
      \infty &>R \geq \norm{x^{m+1}-x^*}^2+2\sum_{k=0}^{m}(f_k-f^*)\alpha_k \geq 2\sum_{k=0}^{m}(f_k-f^*)\alpha_k\\
      &\geq 2\sum_{\ell=0}^{\ell_m}\sum_{j=m_{\ell}}^{n_{\ell}-1}\underbrace{(f_j-f^*)}_{>\tfrac{1}{3}\eta}\alpha_j
      > \tfrac{2}{3}\eta \sum_{\ell=0}^{\ell_m}\sum_{j=m_{\ell}}^{n_{\ell}-1}\alpha_j.
    \end{align*}
    But for $m\to\infty$, this yields a contradiction since the sum on
    the right hand side diverges. Hence, there does not exist an
    accumulation point strictly larger than $f^*$, so we can conclude
    $f_k\to f^*$ as $k\to\infty$, i.e., the whole sequence $(f_k)$
    converges to $f^*$.

    \hspace*{1.5em}We now consider convergence of the sequence $(x^k$).
    From~\eqref{eq:5} we conclude that both terms on the left hand
    side are bounded independently of~$m$. In particular this means
    $(x^k)$ is a bounded sequence. Hence, by the
    Bolzano-Weierstra{\ss} Theorem, it has a convergent subsequence
    $(x^{k_i})$ with $x^{k_i}\to \overline{x}$ (as $i\to\infty$) for
    some $\overline{x}$. To show that the full sequence $(x^k)$
    converges to $\overline{x}$, take any $K$ and any $k_i < K$ and
    observe from~\eqref{eq:basic_isa_divergent_series} that
    \[
    \norm{x^K - \overline{x}}^2 \leq \norm{x^{k_i} - \overline{x}}^2 + \sum_{j=k_i}^{K-1}\beta_j.
    \]
    Since $\sum_k \beta_k$ is a convergent series (as seen from the
    second last line of~\eqref{eq:5}), the right hand side becomes
    arbitrarily small for $k_i$ and~$K$ large enough.  This implies
    $x^k\to \overline{x}$, and since $\varepsilon_k \to 0$, $f_k \to
    f^*$, and $\X^*$ is closed, $\overline{x} \in \X^*$ must hold.
    
  \item[ii)] Now consider the case where $f_k<f^*$ occurs infinitely
    often.  We write $(f^-_k)$ for the subsequence of $(f_k)$ with
    $f_k<f^*$ and $(f^+_k)$ for the subsequence with $f_k\geq
    f^*$. Clearly $f^-_k\to f^*$. Indeed, the corresponding iterates
    are asymptotically feasible (since the projection accuracy
    $\varepsilon_k$ tends to zero), and hence $f^*$ is the only
    possible accumulation point of $(f^-_k)$.
    
    \hspace*{1.5em}Denoting $M^-_m = \{k\leq m\suchthat f_k<f^*\}$ and
    $M^+_m = \{k\leq m\suchthat f_k\geq f^*\}$, we conclude
    from~\eqref{eq:5} that
    \begin{equation}
      \label{eq:est_f_splitted}
      \norm{x^{m+1}-x^*}^2 + 2\sum_{k\in M^+_m} (f_k-f^*)\,\alpha_k\leq R + 2\sum_{k\in
        M^-_m} (f^*-f_k)\,\alpha_k.
    \end{equation}
    Note that each summand is non-negative. To see that the right hand
    side is bounded independently of $m$, let $y^{k-1} =
    x^{k-1}-\alpha_{k-1}h^{k-1}$, and observe that here ($k\in
    M^-_m$), due to $f_k< f^*\leq f(\Pr_X(y^{k-1}))$, we have
    \begin{align*}
      f^*-f_k &\leq f\big(\Pr_X(y^{k-1})\big)-f\big(\Pr_X^{\varepsilon_{k-1}}(y^{k-1})\big)\\
      &\leq (h^{k-1})^\top\big(\Pr_X(y^{k-1})-\Pr_X^{\varepsilon_{k-1}}(y^{k-1})\big)\\
      &\leq \norm{h^{k-1}}\cdot \big\lVert\Pr_X(y^{k-1})-\Pr_X^{\varepsilon_{k-1}}(y^{k-1})\big\rVert \leq H\varepsilon_{k-1},
    \end{align*}
    using the subgradient and Cauchy-Schwarz inequalities as well as
    pro\-perty~\eqref{eq:IPr} of $\Pr_X^\varepsilon$ and the boundedness
    of the subgradient norms. From~\eqref{eq:est_f_splitted}, using
    \eqref{eq:alpha} and \eqref{eq:alphageqepsrest}, we thus obtain
    \begin{align}
      \nonumber &\norm{x^{m+1}-x^*}^2 + 2\sum_{k\in M^+_m} (f_k-f^*)\alpha_k \leq R + 2\,H\sum_{k\in M^-_m} \alpha_k\, \varepsilon_{k-1}\\
      \label{eq:est_f_splitted_var} \leq~ &R + 2\,H\sum_{k\in M^-_m} \alpha_k \alpha_{k-1} \leq R+2H\sum_{k=0}^{\infty}\alpha_k \alpha_{k-1} \leq R+4\,A\,H <\infty.
    \end{align}
    Similar to case i), we conclude that both the sequence $(x^k)$ and
    the series $\sum_{k\in M^+_m} (f_k-f^*)\,\alpha_k$ are bounded.
    
    \hspace*{1.5em}It remains to show that $f^+_k \to f^*$. Assume to
    the contrary that $(f^+_k)$ has an accumulation point $f^*+\eta$
    for $\eta>0$.  Similar to before, we construct index subsequences
    $(m_\ell)$ and $(p_\ell)$ as follows: Set $p_{(-1)}\define -1$ and
    define, for $\ell\geq 0$,
    \begin{align*}
      m_\ell &\define \min\{\,k\in M^+_\infty\suchthat k>p_{\ell-1},\,f_k>f^*+\tfrac{2}{3}\eta\,\},\\
      p_\ell &\define \min\{\,k\in M^-_\infty\suchthat k>m_{\ell}\,\}.
    \end{align*}
    Then $m_{\ell},\dots,p_{\ell}-1\in M^+_\infty$ for all $\ell$, and
    we have
    \[
    \tfrac{2}{3}\eta < f_{m_{\ell}}-f_{p_{\ell}}\leq H^2\sum_{j=m_{\ell}}^{p_{\ell}-1}\alpha_j +H\sum_{j=m_{\ell}}^{p_{\ell}-1}\varepsilon_j .
    \]
    Therefore, with $\ell_m\define\max\{\,\ell\,\vert\,p_{\ell}-1\leq
    m\,\}$ for a given $m$,
    \[
    \tfrac{2}{3}\sum_{\ell=0}^{\ell_m}\eta \leq H^2\sum_{\ell=0}^{\ell_m}\sum_{j=m_{\ell}}^{p_{\ell}-1}\alpha_j+H\sum_{\ell=0}^{\ell_m}\sum_{j=m_{\ell}}^{p_{\ell}-1}\varepsilon_j
    \leq H^2\sum_{\ell=0}^{\ell_m}\sum_{j=m_{\ell}}^{p_{\ell}-1}\alpha_j+H\,E.
    \]
    Now the left hand side becomes arbitrarily large as $m\to\infty$,
    so that also
    $\sum_{\ell=0}^{\ell_m}\sum_{j=m_{\ell}}^{p_{\ell}-1}\alpha_j\to\infty$,
    since $HE<\infty$.  Note that because $\alpha_k >0$ and
    \[
    \sum_{\ell=0}^{\ell_m}\sum_{j=m_{\ell}}^{p_{\ell}-1}\alpha_j\leq\sum_{k\in
      M^+_m}\alpha_k,
    \]
    this latter series must diverge as well. As a consequence, $f^*$
    is itself an (other) accumulation point of $(f^+_k)$:
    From~\eqref{eq:est_f_splitted_var} we have
    \begin{align*}
      \infty &>R+4AH \geq 2\sum_{k\in M^+_m}(f_k-f^*)\alpha_k\\
      &\geq \sum_{k\in M^+_m}(\underbrace{\min\{\,f_j\,\vert\,j\in M^+_m,\,j\leq m\,\}}_{\eqqcolon\hat{f}^*_m}-f^*)\,\alpha_k = (\hat{f}^*_m-f^*)\sum_{k\in M^+_m}\alpha_k ,
    \end{align*}
    and thus 
    \[
    0\leq \hat{f}^*_m-f^*\leq\frac{R + 4\, A\, H}{\sum_{k\in M^+_m}\alpha_k}\to 0\qquad {\rm as}~m\to\infty,
    \]
    since $\sum_{k\in M^+_m}\alpha_k$ diverges. But then, knowing
    $(\hat{f}^*_k)$ converges to $f^*$, we can use $(m_{\ell})$ and
    another index subsequence $(n_{\ell})$, given by
    \[
    n_{\ell} \define \min\{\,k\in M^+_\infty \,\vert\,k>m_{\ell},\,f_k<f^*+\tfrac{1}{3}\eta\,\},
    \]
    to proceed analogously to case i) to arrive at a contradiction and
    conclude that no $\eta>0$ exists such that $f^*+\eta$ is an
    accumulation point of $(f^+_k)$.
    
    \hspace*{1.5em}On the other hand, since $(x^k)$ is bounded and $f$
    is continuous on a neighborhood of $\X$ (recall that for all $k$,
    $x^k$ is contained in an $\varepsilon_k$-neighborhood of $\X$), $(f^+_k)$ is
    bounded. Thus, it must have at least one accumulation point. Since
    $f_k\geq f^*$ for all $k\in M^+_\infty$, the only possibility left
    is $f^*$ itself. Hence, $f^*$ is the unique accumulation point
    (i.e., the limit) of the sequence $(f^+_k)$. As this is also true
    for $(f^-_m)$, the whole sequence $(f_k)$ converges to $f^*$.
    
    \hspace*{1.5em}Finally, convergence of the bounded sequence $(x^k)$
    to some $\overline{x}\in\X^*$ can now be obtained just like in
    case i), completing the proof.  \qed
  \end{enumerate}

\subsection{ISA with dynamic Polyak-type step sizes}\label{sect:ISA_conv_dynamic}

Let us now turn to dynamic step sizes. In the rest of this section, $\alpha_k$ will always denote
step sizes of the form \eqref{eq:alphak}.

Since in subgradient methods the objective function values need not
decrease monotonically, the key quantity in convergence proofs usually
is the distance to the optimal set $\X^*$. For ISA with dynamic
step sizes (Algorithm \ref{alg:DynamicISA}), we have the following
result concerning these distances:

\begin{lemma}\label{lem:1}
  Let $x^*\in\X^*$. For the sequence of ISA iterates $(x^k)$, computed
  with step sizes $\alpha_k=\lambda_k(f_k-\varphi)/\norm{h^k}^2$, it
  holds that
  \begin{align}
    \nonumber\lVert x^{k+1}-x^*\rVert^2~\leq~&\lVert x^k-x^*\rVert^2+\varepsilon_k^2+2\left(\frac{\lambda_k(f_k-\varphi)}{\norm{h^k}}+\lVert x^k-x^*\rVert\right)\varepsilon_k\\
    \label{eq:lem1}&+\frac{\lambda_k(f_k-\varphi)}{\norm{h^k}^2}\big(\lambda_k(f_k-\varphi)-2(f_k-f^*)\big).    
  \end{align}
  In particular, also
  \begin{equation}\label{eq:distXstar}
    \d_{X^*}(x^{k+1})^2\leq\d_{X^*}(x^k)^2-2\,\alpha_k(f_k-f^*)+(\alpha_k\norm{h^k} + \varepsilon_k)^2+2\,\d_{X^*}(x^k)\,\varepsilon_k.
  \end{equation}
\end{lemma}
\begin{proof}
  Plug~\eqref{eq:alphak} into~\eqref{eq:distIPrxSq} for $x=x^*$ and
  rearrange terms to obtain (\ref{eq:lem1}). If the optimization
  problem (\ref{eq:NLO}) has a unique optimum $x^*$, then obviously
  $\lVert x^k-x^*\rVert=\d_{X^*}(x^k)$ for all $k$, so
  \eqref{eq:distXstar} is identical to~\eqref{eq:lem1}. Otherwise,
  note that since $X^*$ is the intersection of the closed set~$\X$
  with the level set $\{x \suchthat f(x) = f^*\}$ of the convex
  function $f$, $X^*$ is closed (cf., for example,
  \cite[Prop.~1.2.2,~1.2.6]{HUL04}) and the projection onto $X^*$ is
  well-defined. Then, considering $x^*=\Pr_{\X^*}(x^k)$,
  \eqref{eq:distIPrxSq} becomes
  \[
    \big\lVert x^{k+1}-\Pr_{\X^*}(x^{k})\big\rVert^2\leq\d_{X^*}(x^k)^2-2\alpha_k(f_k-f^*)+(\alpha_k\norm{h^k}+\varepsilon_k)^2+2\,\d_{X^*}(x^k)\,\varepsilon_k.
  \]
  Furthermore, because obviously $f(\Pr_{X^*}(x))=f(\Pr_{X^*}(y))=f^*$
  for all $x,y\in\R^n$, and by definition of the Euclidean projection,
  \[
    \d_{X^*}(x^{k+1})^2~=~\big\lVert
    x^{k+1}-\Pr_{\X^*}(x^{k+1})\big\rVert^2~\leq~ \big\lVert x^{k+1}-\Pr_{\X^*}(x^{k})\big\rVert^2.
  \]
  Combining the last two inequalities yields~\eqref{eq:distXstar}. 

  Moreover, note that these results continue to hold true if $x^{k+1}$
  is replaced in a projection refinement phase (starting in the next
  iteration $k+1$), since then only accuracy parameters smaller than
  $\varepsilon_k$ are used.\qed
\end{proof}

%\noindent
Typical convergence results are often derived by showing that the
sequence $(\norm{x^k-x^*})$ is monotonically decreasing (for arbitrary
$x^*\in\X^*$) under certain assumptions on the step sizes,
subgradients, etc. This is also done in~\cite{AHKS87},
where~(\ref{eq:lem1}) with $\varepsilon_k=0$ for all $k$ is the
central inequality, cf.~\cite[Prop.~2]{AHKS87}.  In our case, i.e.,
working with adaptive approximate projections as specified by
(\ref{eq:IPr}), we can follow this principle to derive conditions on
the projection accuracies $(\varepsilon_k)$ which still allow for a
(monotonic) decrease of the distances from the optimal set: If the
last summand in (\ref{eq:lem1}) is negative, the resulting gap between
the distances from $\X^*$ of subsequent iterates can be exploited to
relax the projection accuracy, i.e., to choose $\varepsilon_k>0$
without destroying monotonicity.

Naturally, to achieve feasibility (at least in the limit), we will
need to have~$(\varepsilon_k)$ diminishing ($\varepsilon_k\to 0$ as
$k\to\infty$).  It will become clear that this, combined with
summability ($\sum_{k=0}^\infty \varepsilon_k < \infty$) and with
monotonicity conditions as described above, is already enough to
extend the analysis to cover iterations with $f_k < f^*$, which may
occur since we project inaccurately.

For different choices of the estimate $\varphi$ of $f^*$, we will now
derive the proofs of Theorems~\ref{thm:ISA_conv_over}
and~\ref{thm:ISA_conv_under} via a series of intermediate
results. Corresponding results for exact projections
($\varepsilon_k=0$) can be found in \cite{AHKS87}. In fact, our
analysis for adaptive approximate projections improves on some of
these earlier results (e.g., \cite[Prop. 10]{AHKS87} states
convergence of some \emph{sub}sequence of the function values to the
optimum for the case $\varphi<f^*$, whereas
Theorem~\ref{thm:ISA_conv_under} in this paper gives convergence of
the whole sequence $(f_k)$, for approximate and also for exact
projections).

For the remainder of this section we can assume that ISA does not
terminate in Step~\ref{step:terminate_opt} and that all inner
projection accuracy refinement loops are finite. Otherwise, there is
some refinement phase starting at iteration $\underline{k}$ such that,
as $\ell\to\infty$, $x^{\underline{k}}$ is repeatedly reset to
\[
y_{\underline{k}}^\ell\define\Pr_\X^{\gamma^\ell
  \varepsilon_{\underline{k}-1}}(x^{\underline{k}-1}-\alpha_{\underline{k}-1}h^{\underline{k}-1})\to\Pr_X^0(x^{\underline{k}-1}-\alpha_{\underline{k}-1}h^{\underline{k}-1})\in\X,
\]
with $f(y_{\underline{k}}^\ell)\to\underline{\varphi}\leq\varphi$;
cf. Remarks~\ref{rem:ISA_conv_over} and \ref{rem:ISA_conv_under}.
  
\subsubsection{Using overestimates of the optimal value.}

In this part we will focus on the case $\varphi\geq f^*$. As might be
expected, this relation allows for eliminating the unknown $f^*$
from~\eqref{eq:distXstar}.
\begin{lemma}\label{lem:2}
  Let $\varphi\geq f^*$ and $\lambda_k \geq 0$. If $f_k\geq\varphi$ for
  some $k\in\N$, then
  \begin{align}
    \nonumber\d_{\X^*}(x^{k+1})^2~\leq~&\d_{\X^*}(x^k)^2+\varepsilon_k^2+2\left(\frac{\lambda_k(f_k-\varphi)}{\norm{h^k}}+\d_{\X^*}(x^k)\right)\varepsilon_k\\
    \label{eq:lem2}&+\frac{\lambda_k(\lambda_k-2)(f_k-\varphi)^2}{\norm{h^k}^2}.
  \end{align}
\end{lemma}
\begin{proof}
  This follows immediately from Lemma~\ref{lem:1}, using
  $f_k\geq\varphi\geq f^*$ and $\lambda_k\geq 0$.\qed
\end{proof}

Note that ISA guarantees $f_k > \varphi$ by sufficiently accurate
projection (otherwise the method stops or the inner refinement loop
over $\ell$, with fixed $k$, is infinite, indicating $\varphi$ was too
large, see Steps~\ref{step:subgradient}-\ref{step:refine_proj} of
Algorithm~\ref{alg:DynamicISA}), and that the last summand
in~\eqref{eq:lem2} is always negative for $0 < \lambda_k < 2$. Hence,
adaptive approximate projections ($\varepsilon_k> 0$) can always be
employed without destroying the monotonic decrease of
$(\d_{\X^*}(x^k))$, as long as the $\varepsilon_k$ are chosen small
enough.

The following result provides a theoretical bound on how large the
projection accuracies $\varepsilon_k$ may become.
\begin{lemma}\label{lem:3}
  Let $0<\lambda_k<2$ for all $k$. For $\varphi\geq f^*$, the sequence
  $(\d_{X^*}(x^k))$ is monotonically decreasing and converges to some
  $\zeta\geq 0$, if $0 \leq \varepsilon_k \leq \overline{\varepsilon}_k$
  for all $k$, where $\overline{\varepsilon}_k$ is defined
  in~\eqref{eq:epsISAover} of Theorem~\ref{thm:ISA_conv_over}.
\end{lemma}

\begin{proof}
  Considering (\ref{eq:lem2}), it suffices to show that for $\varepsilon_k
  \leq \overline{\varepsilon}_k$, we have
  \begin{equation}\label{eq:lem3proof}
    \varepsilon_k^2+2\left(\frac{\lambda_k(f_k-\varphi)}{\norm{h^k}}+\d_{X^*}(x^k)\right)\varepsilon_k+\frac{\lambda_k(\lambda_k-2)(f_k-\varphi)^2}{\norm{h^k}^2}\leq 0.
  \end{equation}
  The bound $\overline{\varepsilon}_k$ from~\eqref{eq:epsISAover} is
  precisely the (unique) positive root of the quadratic function in
  $\varepsilon_k$ given by the left hand side of~\eqref{eq:lem3proof}.
  Thus, we have a monotonically decreasing (i.e., nonincreasing)
  sequence $(\d_{X^*}(x^k))$, and since its members are bounded below
  by zero, it converges to some nonnegative value, say $\zeta$.\qed
\end{proof}

As a consequence, if $X^*$ is bounded, we obtain boundedness of the
iterate sequence $(x^k)$:
\begin{corollary}\label{cor:1}
  Let $X^*$ be bounded. If the sequence $(\d_{\X^*}(x^k))$ is
  monotonically decreasing, then the sequence $(x^k)$ is bounded.
\end{corollary}
\begin{proof}
  By monotonicity of $(\d_{X^*}(x^k))$, making use of the triangle
  inequality,
  \begin{align*}
    \lVert x^k\rVert~&=~ \big\lVert x^k-\Pr_{X^*}(x^k)+\Pr_{X^*}(x^k)\big\rVert\\
    &\leq~\d_{X^*}(x^k) + \big\lVert \Pr_{X^*}(x^k)\big\rVert~\leq~\d_{X^*}(x^0)+\sup_{x\in\X^*}\lVert x\rVert~<~\infty,
  \end{align*}
  since $X^*$ is bounded by assumption.\qed
\end{proof}

We now have all the tools at hand for proving
Theorem~\ref{thm:ISA_conv_over}.\\

\noindent
\textbf{Proof of Theorem \ref{thm:ISA_conv_over}.}  First, we prove
part (i). Let the main assumptions of Theorem~\ref{thm:ISA_conv_over}
hold and suppose---contrary to the desired result (i)---that
$f_k>\varphi+\delta$ for all $k$ (possibly after finitely many
refinements of the projection accuracy used to compute $x^k$). By
Lemma~\ref{lem:2},
   \begin{align*}
     \frac{\lambda_k(2-\lambda_k)(f_k-\varphi)^2}{\norm{h^k}^2}~\leq~~&\d_{X^*}(x^k)^2-\d_{X^*}(x^{k+1})^2\\
     &+\varepsilon_k^2+2\left(\frac{\lambda_k(f_k-\varphi)}{\norm{h^k}}+\d_{X^*}(x^k)\right)\varepsilon_k.
   \end{align*}
   Since $0 < \underline{H} \leq \norm{h^k} \leq \overline{H}<\infty$, $0 < \lambda_k
   \leq \beta < 2$, and $f_k-\varphi>\delta$ for all $k$ by
   assumption, we have
   \[
   \frac{\lambda_k(2-\lambda_k)(f_k-\varphi)^2}{\norm{h^k}^2}~\geq~\frac{\lambda_k(2-\beta)\delta^2}{\overline{H}^2}.
   \]
   By Lemma~\ref{lem:3}, $\d_{X^*}(x^k)\leq\d_{X^*}(x^0)$. Also, by
   Corollary~\ref{cor:1} there exists $F < \infty$ such that $f_k\leq
   F$ for all $k$. Hence,
   $\lambda_k(f_k-\varphi)\leq\beta(F-\varphi)$, and since
   $1/\norm{h^k}\leq 1/\underline{H}$, we obtain
   \begin{equation}\label{eq:thm1proof}
     \frac{(2-\beta)\delta^2}{\overline{H}^2}\lambda_k\leq\d_{X^*}(x^k)^2-\d_{X^*}(x^{k+1})^2+\varepsilon_k^2+2\left(\frac{\beta(F-\varphi)}{\underline{H}}+\d_{X^*}(x^0)\right)\varepsilon_k.
   \end{equation}
   Summation of the inequalities~\eqref{eq:thm1proof} for
   $k=0,1,\dots,m$ yields
   \begin{align*}
     \frac{(2-\beta)\delta^2}{\overline{H}^2}\sum_{k=0}^{m}\lambda_k \leq~ & \d_{X^*}(x^0)^2-\d_{X^*}(x^{m+1})^2\\
     &+\sum_{k=0}^{m}\varepsilon_k^2+2\left(\frac{\beta(F-\varphi)}{\underline{H}}+\d_{X^*}(x^0)\right)\sum_{k=0}^{m}\varepsilon_k.
   \end{align*}
   Now, by assumption, the left hand side tends to infinity as $m \to
   \infty$, while the right hand side remains finite (note that
   nonnegativity and summability of $(\nu_k)$ imply the summability of
   $(\nu_k^2)$, properties that carry over to
   $(\varepsilon_k)$). Thus, we have reached a contradiction and
   therefore proven part (i) of Theorem \ref{thm:ISA_conv_over}, i.e.,
   that $f_K\leq\varphi+\delta$ holds in some iteration $K$.

   We now turn to part (ii): Let the main assumptions of Theorem
   \ref{thm:ISA_conv_over} hold, let $\lambda_k\to 0$ and suppose
   $f_k>\varphi$ for all $k$ (again, possibly after
   refinements). Then, since we know from part (i) that the function
   values fall below every $\varphi+\delta$, we can construct a
   monotonically decreasing subsequence $(f_{K_j})$ such that
   $f_{K_j}\to \varphi$. (To see this, note that if
   $f_k<\varphi+\delta$ is reached with $f_k<\varphi$, the ensuing
   refinement phase not necessarily ends with $x^k$ replaced by a
   point with $\varphi<f_k<\varphi+\delta$, but that then, however,
   there always exists a $K>k$ such that $\varphi<f_K<\varphi+\delta$,
   since $\lambda_k\to 0$, $\varepsilon_k\to 0$, and by continuity of
   $f$.)

   To show that $\varphi$ is the unique accumulation point of $(f_k)$,
   assume to the contrary that there is another subsequence of $(f_k)$
   which converges to $\varphi +\eta$, with some $\eta>0$. We can now
   employ the same technique as in the proof of Theorem
   \ref{thm:ISA_conv_apriori} to reach a contradiction:

   The two cases $f_k<\varphi+\tfrac{1}{3}\eta$ and
   $f_k>\varphi+\tfrac{2}{3}\eta$ must both occur infinitely often,
   since $\varphi$ and $\varphi+\eta$ are accumulation points. Set
   $n_{(-1)}\define-1$ and define, for $\ell\geq 0$,
   \begin{align*}
     m_\ell &\define \min\{\,k\suchthat k>n_{\ell-1},\,f_k>\varphi+\tfrac{2}{3}\eta\,\},\\
     n_\ell &\define \min\{\,k\suchthat k>m_{\ell},\,f_k<\varphi+\tfrac{1}{3}\eta\,\}.
   \end{align*}
   Then, with $\infty>F\geq f_k$ for all $k$ (existing since $(x^k)$
   is bounded and therefore so is $(f_k)$) and the subgradient norm
   bounds, we obtain
   \[
   \tfrac{1}{3}\eta < f_{m_{\ell}}-f_{n_{\ell}} \leq \overline{H}\norm{x^{m_{\ell}}-x^{n_{\ell}}} \leq \frac{\overline{H}(F-\varphi)}{\underline{H}}\sum_{j=m_{\ell}}^{n_{\ell}-1}\lambda_j +\overline{H}\sum_{j=m_{\ell}}^{n_{\ell}-1}\varepsilon_j
   \]
   and from this, denoting $\ell_m\define\max\{\,\ell\suchthat n_{\ell}-1\leq m\,\}$ for a given $m$,
   \[
   \tfrac{1}{3}\sum_{\ell=0}^{\ell_{m}}\eta \leq \frac{\overline{H}(F-\varphi)}{\underline{H}}\sum_{\ell=0}^{\ell_{m}}\sum_{j=m_{\ell}}^{n_{\ell}-1}\lambda_j +\overline{H}\sum_{\ell=0}^{\ell_{m}}\sum_{j=m_{\ell}}^{n_{\ell}-1}\varepsilon_j.
   \]
   Since for $m\to\infty$, the left hand side tends to infinity, the
   same must hold for the right hand side. But since
   $\sum_{\ell=0}^{\ell_{m}}\sum_{j=m_\ell}^{n_\ell-1}\varepsilon_j\leq
   \sum_{k=0}^{m}\varepsilon_k\leq \sum_{k=0}^{m}\nu_k<\infty$, this
   implies
   \begin{equation}\label{eq:proof_2ii_divser}
     \sum_{\ell=0}^{\ell_{m}}\sum_{j=m_{\ell}}^{n_{\ell}-1}\lambda_j\to\infty
     \qquad\text{for }m\to\infty.
   \end{equation}
   
   Also, using the same estimates as in part (i) above,
   (\ref{eq:lem2}) yields
   \begin{equation}\nonumber
     \underbrace{\tfrac{2-\beta}{\overline{H}}}_{\eqqcolon C_1<\infty}(f_k-\varphi)^2\lambda_k \leq d_{X^*}(x^k)^2-d_{X^*}(x^{k+1})^2+\varepsilon_k^2+\underbrace{2\left(\tfrac{\beta(F-\varphi)}{\underline{H}}+d_{X^*}(x^0)\right)}_{\eqqcolon C_2<\infty}\varepsilon_k
   \end{equation}
   and thus by summation for $k=0,\dots,m$ for a given $m$,
   \begin{equation}\label{eq:proof_2ii}
     C_1 \sum_{k=0}^{m}(f_k-\varphi)^2\lambda_k \leq d_{X^*}(x^0)^2 -d_{X^*}(x^{m+1})^2 +\sum_{k=0}^{m}\varepsilon_k^2+C_2\sum_{k=0}^{m}\varepsilon_k.
   \end{equation}
   Observe that all summands of the left hand side term are positive, and thus
   \[
     C_1 \sum_{k=0}^{m}(f_k-\varphi)^2\lambda_k \geq C_1 \sum_{\ell=0}^{{\ell}_ {m}}\sum_{j=m_{\ell}}^{n_{\ell}-1}(\underbrace{f_j-\varphi}_{>\tfrac{1}{3}\eta})^2\lambda_j > \frac{C_1 \eta^2}{9} \sum_{\ell=0}^{{\ell}_ {m}}\sum_{j=m_{\ell}}^{n_{\ell}-1}\lambda_j.
   \]
   Therefore, as $m\to\infty$, the left hand side of
   (\ref{eq:proof_2ii}) tends to infinity (by
   (\ref{eq:proof_2ii_divser}) and the above inequality) while the
   right hand side expression remains finite (recall
   $0\leq\varepsilon_k\leq\nu_k$ with $(\nu_k)$ summable and thus also
   square-summable). Thus, we have reached a contradiction, and it
   follows that $\varphi$ is the only accumulation point (i.e., the
   limit) of the whole sequence $(f_k)$.

   This proves part (ii) and thus completes the proof of Theorem \ref{thm:ISA_conv_over}.\qed

\begin{remark}\label{rem:conv_x_overest}
  With more technical effort one can argue along the lines of the proof of
  Theorem~\ref{thm:ISA_conv_apriori} to obtain the following result on the
  convergence of the iterates~$x^k$ in the case of
  Theorem~\ref{rem:ISA_conv_over}: If we additionally assume that $\sum
  \lambda_k^2<\infty$ and that $\lambda_k\geq \sum_{j=k}^\infty \varepsilon_k$
  for all $k$, then $x^k \to \overline{x}$ for some $\overline{x} \in X$
  with $f(\overline{x})=\varphi$ and $d_{X^*}(\overline{x}) = \zeta\geq 0$ ($\zeta$
  being the same as in Lemma~\ref{lem:3}).
\end{remark}

\subsubsection{Using lower bounds on the optimal value.}

In the following, we focus on the case $\varphi<f^*$, i.e., using a
constant lower bound in the step size definition
(\ref{eq:alphak}). Such a lower bound is often more readily available
than (useful) upper bounds; for instance, it can be computed via the
dual problem, or sometimes derived directly from properties of the
objective function such as, e.g., nonnegativity of the function
values.

Following arguments similar to those in the previous part, we can
prove convergence of ISA (under certain assumptions), provided
that the projection accuracies $(\varepsilon_k)$ obey conditions
analogous to those for the case $\varphi\geq f^*$. Moreover, recall
that for $\varphi<f^*$, every refinement phase is finite, so that
$f_k>\varphi$ is guaranteed for all $k$; in particular,
Step~\ref{step:terminate_opt} is never executed since $X\cap\{x \suchthat
f(x)<\varphi\}=\emptyset$.

Let us start with analogues of Lemmas~\ref{lem:2} and~\ref{lem:3}.

\begin{lemma}\label{lem:4}
  Let $\varphi<f^*$ and $0 < \lambda_k \leq \beta < 2$. If $f_k \geq
  \varphi$ for some $k \in \N$, then
  \begin{equation}\label{eq:lem4}
    \d_{X^*}(x^{k+1})^2~\leq~\d_{X^*}(x^k)^2+\varepsilon_k^2+2\left(\frac{\lambda_k(f_k-\varphi)}{\norm{h^k}}+\d_{X^*}(x^k)\right)\varepsilon_k+L_k,
  \end{equation}
  where $L_k$ is defined in~\eqref{eq:LkDef} of
  Theorem~\ref{thm:ISA_conv_under}.
\end{lemma}

\begin{proof}
  For $\varphi<f^*$, $0<\lambda_k\leq\beta<2$, and $f_k\geq\varphi$,
  it holds that
  \[
  \lambda_k(f_k-\varphi)-2(f_k-f^*)~\leq\beta(f_k-\varphi)-2(f_k-f^*)=\beta(f^*-\varphi)+(2-\beta)(f^*-f_k).
  \]
  The claim now follows immediately from Lemma \ref{lem:1}.\qed
\end{proof}

\begin{lemma}\label{lem:5}
  Let $\varphi<f^*$, let $0<\lambda_k\leq\beta<2$ and $f_k\geq
  f^*+\tfrac{\beta}{2-\beta}(f^*-\varphi)$ for all~$k$, and let $L_k$
  be given by~\eqref{eq:LkDef}. Then $(\d_{X^*}(x^k))$ is
  monotonically decreasing and converges to some $\xi \geq 0$, if $0
  \leq \varepsilon_k \leq \tilde{\varepsilon}_k$ for all $k$, where
  $\tilde{\varepsilon}_k$ is defined in~\eqref{eq:epsISAunder}.
\end{lemma}

\begin{proof}
  The condition $f_k\geq f^*+\tfrac{\beta}{2-\beta}(f^*-\varphi)$
  implies $L_k \leq 0$ and hence ensures that adaptive approximate projection can
  be used while still allowing for a decrease in the distances of the
  subsequent iterates from $\X^*$. The rest of the proof is completely
  analogous to that of Lemma~\ref{lem:3}, considering~\eqref{eq:lem4}
  and~\eqref{eq:LkDef} to derive the upper bound
  $\tilde{\varepsilon}_k$ given by~\eqref{eq:epsISAunder} on the
  projection accuracy.\qed
\end{proof}

We can now state the proof of our convergence results for the case
$\varphi<f^*$.\\

%\noindent
\textbf{Proof of Theorem \ref{thm:ISA_conv_under}.} Let the
assumptions of Theorem \ref{thm:ISA_conv_under} hold. We start with
proving part (i): Let some $\delta>0$ be given and suppose---contrary
to the desired result (i)---that $f_k >
f^*+\tfrac{\beta}{2-\beta}(f^*-\varphi)+\delta$ for all $k$ (possibly
after refinements). By Lemma~\ref{lem:4},
   \[
   \d_{X^*}(x^{k+1})^2~\leq~\d_{X^*}(x^k)^2+\varepsilon_k^2+2\left(\frac{\lambda_k(f_k-\varphi)}{\norm{h^k}}+\d_{X^*}(x^k)\right)\varepsilon_k+L_k.
   \]
   Since $0 < \underline{H} \leq \norm{h^k} \leq \overline{H}$, $0 < \lambda_k \leq
   \beta < 2$, and $\varphi< f_k$, and due to our assumption on $f_k$,
   i.e.,
   \[
   f^*-f_k+\tfrac{\beta}{2-\beta}(f^*-\varphi)~<~-\delta\qquad\text{for all }k,
   \]
   it follows that
   \[
   L_k~<~-\,\frac{\lambda_k(2-\beta)(f_k-\varphi)\delta}{\overline{H}^2}~<~0.
   \]
   By Lemma~\ref{lem:5}, $\d_{X^*}(x^k)\leq\d_{X^*}(x^0)$, and
   Corollary~\ref{cor:1} again ensures existence of some $F<\infty$
   such that $f_k\leq F$ for all $k$.  Because also
   $\lambda_k(f_k-\varphi)\leq\beta(F-\varphi)$ and $1/\norm{h^k} \leq
   1/\underline{H}$, we hence obtain
   \begin{align}
     \nonumber \frac{\lambda_k(2-\beta)(f_k-\varphi)\delta}{\overline{H}^2}~<~-L_k
     ~\leq~ & \d_{X^*}(x^k)^2-\d_{X^*}(x^{k+1})^2\\
     & + \varepsilon_k^2+2\left(\frac{\beta(F-\varphi)}{\underline{H}}+\d_{X^*}(x^0)\right)\varepsilon_k.\label{eq:thm3_sum_lambda_est}
   \end{align}
   Summation of these inequalities for $k=0,1,\dots,m$ yields
   \begin{align}
     \nonumber \frac{(2-\beta)\delta}{\overline{H}^2}\sum_{k=0}^{m}(f_k-\varphi)\lambda_k ~<~ & \d_{X^*}(x^0)^2-\d_{X^*}(x^{m+1})^2\\
     & + \sum_{k=0}^{m}\varepsilon_k^2+2\left(\frac{\beta(F-\varphi)}{\underline{H}}+\d_{X^*}(x^0)\right)\sum_{k=0}^{m}\varepsilon_k.
     \label{eq:thm3_sum_lambda_more}
   \end{align}
   Moreover, our assumption on $f_k$ yields   
   \[
   f_k-\varphi~>~f^*+\tfrac{\beta}{2-\beta}f^*-\tfrac{\beta}{2-\beta}\varphi+\delta-\varphi~=~\tfrac{2}{2-\beta}(f^*-\varphi)+\delta.
   \]
   It follows from~\eqref{eq:thm3_sum_lambda_more} that
   \begin{align*}
     \frac{\big(2(f^*-\varphi)+(2-\beta)\delta\big)\delta}{\overline{H}^2}\sum_{k=0}^{m}\lambda_k <\; &\d_{X^*}(x^0)^2-\d_{X^*}(x^{m+1})^2\\
     & + \sum_{k=0}^{m}\varepsilon_k^2+2\left(\frac{\beta(F-\varphi)}{\underline{H}}+\d_{X^*}(x^0)\right)\sum_{k=0}^{m}\varepsilon_k.
   \end{align*}
   Now, by assumption, the left hand side tends to infinity as
   $m\to\infty$, whereas by Lemma~\ref{lem:5} and the choice of $0
   \leq \varepsilon_k \leq \min\{|\tilde{\varepsilon}_k|,\nu_k\}$ with
   a nonnegative summable (and hence also square-summable) sequence
   $(\nu_k)$, the right hand side remains finite. Thus, we have
   reached a contradiction, and part (i) is proven, i.e., there does
   exist some $K$ such that $f_K\leq
   f^*+\frac{\beta}{2-\beta}(f^*-\varphi)+\delta$ (after possible
   refinements of the projection accuracy used to recompute $x^K$).

   Let us now turn to part (ii): Again, let the main assumptions of
   Theorem~\ref{thm:ISA_conv_under} hold and let $\lambda_k\to 0$.
   Recall that for $\varphi<f^*$, we have $f_k>\varphi$ for all~$k$ by
   construction of ISA (refinement loops). We distinguish three
   cases:

   If $f_k<f^*$ holds for all $k\geq k_0$ for some $k_0$, then $f_k\to
   f^*$ is obtained immediately, just like in the proof of
   Theorem~\ref{thm:ISA_conv_apriori}, by asymptotic feasibility. 
   
   On the other hand, if $f_k\geq f^*$ for all $k$ larger than some
   $k_0$, then repeated application of part (i) yields a subsequence
   of $(f_k)$ which converges to $f^*$: For any $\delta > 0$ we can
   find an index $K$ such that $f^*\leq f_K \leq f^* +
   \tfrac{\beta}{2-\beta}(f^*-\varphi) + \delta$. Obviously, we get
   arbitrarily close to $f^*$ if we choose $\beta$ and $\delta$ small
   enough. However, we have the restriction $\lambda_k\leq \beta$. But
   since $\lambda_k\to 0$, we may ``restart'' our argumentation if
   $\lambda_k$ is small enough and replace~$\beta$ with a smaller
   one. With the convergent subsequence thus constructed, we can then
   use the same technique as in the proof of
   Theorem~\ref{thm:ISA_conv_over} (ii) to show that $(f_k)$ has no
   other accumulation point but $f^*$, whence $f_k\to f^*$
   follows. 

   Finally, when both cases $f_k<f^*$ and $f_k\geq f^*$ occur
   infinitely often, we can proceed similar to the proof of
   Theorem~\ref{thm:ISA_conv_apriori}: The subsequence of function
   values below $f^*$ converges to $f^*$, since $\varepsilon_k\to 0$.
   For the function values greater or equal to $f^*$, we assume that
   there is an accumulation point $f^*+\eta$ larger than $f^*$, deduce
   that an appropriate sub-sum of the $\lambda_k$'s diverge and then
   sum up equation \eqref{eq:thm3_sum_lambda_est} for the respective
   indices (belonging to $\{k\suchthat f_k\geq f^*\}$) to arrive at a
   contradiction. Note that the iterate sequence $(x^k)$ is bounded,
   due to Corollary~\ref{cor:1} (for iterations $k$ with $f_k\geq
   f^*$) and since the iterates with $\varphi<f_k<f^*$ stay within a
   bounded neighborhood of the bounded set $X^*$, since
   $\varepsilon_k$ tends to zero and is summable. Therefore, as $f$ is
   continuous on a neighborhood of $\X$ (which contains all~$x^k$
   from some $k$ on), $(f_k)$ is
   bounded as well and therefore must have at least one accumulation
   point. The only possibility left now is $f^*$, so we conclude
   $f_k\to f^*$.\qed
   
\begin{remark}
  With $f_k\to f^*$ and $\varepsilon_k\to 0$, we obviously have
  $d_{X^*}(x^k)\to 0$ in the setting of
  Theorem~\ref{thm:ISA_conv_under}. Furthermore,
  Remark~\ref{rem:conv_x_overest} applies similarly: With more
  conditions on $\lambda_k$ and more technical effort one can obtain
  convergence of the sequence $(x^k)$ to some $\overline{x}\in X^*$.
\end{remark}

\section{Discussion}

In this section, we will discuss extensions of ISA. We will also
illustrate how to obtain bounds on the projection accuracies that are
independent of the (generally unknown) distance from the optimal set,
and thus computable.

\subsection{Extension to $\epsilon$-subgradients}

It is noteworthy that the above convergence analyses also work when
replacing the subgradients by $\epsilon$-subgradients \cite{BM73},
i.e., replacing $\partial f(x^k)$ by
\begin{equation}\label{eq:epssubdiff}
  \partial_{\rho_k} f(x^k) \define \{\,h\in\R^n \suchthat f(x)-f(x^k) \geq
  h^\top (x-x^k) - \rho_k \quad \forall\, x\in\R^n\,\}.
\end{equation}
(To avoid confusion with the projection accuracy parameters
$\varepsilon_k$, we use $\rho_k$.) For instance, we immediately
obtain the following result:
\begin{corollary}\label{cor:epssubgradapriori}
  Let ISA (Algorithm \ref{alg:APrioriISA}) choose
  $h^k\in\partial_{\rho_k}f(x^k)$ with $\rho_k\geq 0$ for all
  $k$. Under the assumptions of Theorem \ref{thm:ISA_conv_apriori}, if
  $(\rho_k)$ is chosen summable $(\sum_{k=0}^\infty
  \rho_k<\infty)$ and such that
  \begin{enumerate}
  \item[\emph{~(i)}] $\rho_k\leq\mu\,\alpha_k$ for some $\mu>0$, or
  \item[\emph{(ii)}] $\rho_k\leq\mu\,\varepsilon_k$ for some $\mu>0$,
  \end{enumerate}
  then the sequence of ISA iterates $(x^k)$ converges to an optimal
  point.
\end{corollary}

\begin{proof} 
  The proof is analogous to that of
  Theorem~\ref{thm:ISA_conv_apriori}; we will therefore only sketch
  the necessary modifications: Choosing
  $h^k\in\partial_{\rho_k}f(x^k)$ (instead of $h^k\in\partial f(x^k)$)
  adds the term $+2\alpha_k\rho_k$ to the right hand side
  of~\eqref{eq:distIPrxSq}.  If $\rho_k\leq \mu\,\alpha_k$ for some
  constant $\mu>0$, the square-summability of $(\alpha_k)$ suffices:
  By upper bounding $2\alpha_k\rho_k$, the constant term $+2\mu A$ is
  added to the definition of $R$ in~\eqref{eq:5}.  Similarly,
  $\rho_k\leq \mu\,\varepsilon_k$ does not impair convergence under
  the assumptions of Theorem \ref{thm:ISA_conv_apriori}, because then
  the additional summand in~\eqref{eq:5} is
  \[
  2\sum_{k=0}^m \alpha_k\rho_k\,\leq\, 2\mu\sum_{k=0}^m
  \alpha_k\varepsilon_k\,\leq\,
  2\mu\sum_{k=0}^m\Big(\alpha_k\sum_{\ell=k}^\infty\varepsilon_k\Big)\,\leq\,
  2\mu\sum_{k=0}^m\alpha_k^2\,\leq\, 2\mu A.
  \]
  The rest of the proof is almost identical, using $R$ modified as
  explained above and some other minor changes where $\rho_k$-terms
  need to be considered, e.g., the term $+\rho_{m_{\ell}}$ is
  introduced in \eqref{eq:est_fdiff_eta_i}, yielding an additional sum
  in \eqref{eq:est_fdiff_eta_i_sum}, which remains finite when passing
  to the limit because $(\rho_k)$ is summable.\qed
\end{proof}

Similar extensions are possible when using dynamic step sizes of the
form~\eqref{eq:alphak}. The upper bounds~\eqref{eq:epsISAover}
and~\eqref{eq:epsISAunder} for the projection accuracies
$(\varepsilon_k)$ will depend on~$(\rho_k)$ as well, which of course
must be taken into account when extending the proofs
accordingly. Then, summability of $(\rho_k)$ (implying $\rho_k\to 0$)
is enough to guarantee convergence.  In particular, one can again
choose $\rho_k\leq\mu\,\varepsilon_k$ for some $\mu>0$. We will not go
into detail here, since the extensions are straightforward.

\subsection{Variable target values}\label{subsect:VTVM}

From a practical viewpoint, it may be desirable to have an algorithm,
using dynamic step sizes, that does not require the user to
\emph{know} a priori whether an estimate $\varphi$ is larger or
smaller than $f^*$, respectively. Moreover, relying on a constant
estimate may lead to overly small or large steps, which slows down the
convergence process (and, w.r.t. ISA (Algorithm~\ref{alg:DynamicISA}),
can also lead to many projection accuracy refinement phases). Thus, a
typical approach is to replace the constant estimate $\varphi$ by
\emph{variable target values} $\varphi_k$. These target values are
then updated in the course of the algorithm to increasingly better
estimates of $f^*$, so that the dynamic step size~\eqref{eq:alphak}
more and more resembles the ``ideal'' Polyak step size (which would
use $\varphi=f^*$). %This extension is also
%possible for the ISA framework, as we will briefly discuss now.
In principle, such extensions are also possible for the ISA framework.
We briefly describe the most important aspects in the following.

First, note that Theorems~\ref{thm:ISA_conv_over} and
\ref{thm:ISA_conv_under} provide bounds on the projection accuracies
$(\varepsilon_k)$ needed for convergence; clearly, if it is unknown
whether $\varphi_k\geq f^*$ or $\varphi_k<f^*$, one must therefore
choose
$0\leq\varepsilon_k\leq\min\{\bar{\varepsilon}_k,\abs{\tilde{\varepsilon}_k},\nu_k\}$,
with $\bar{\varepsilon}_k$ and $\tilde{\varepsilon}_k$ given
by~\eqref{eq:epsISAover} and \eqref{eq:epsISAunder}, respectively.

Crucial for any variable target value method is the ability to somehow
recognize whether $\varphi_k\geq f^*$ or $\varphi_k<f^*$. If all
iterates are feasible, this amounts to recognizing whether $X\cap\{x
\suchthat f(x)\leq\varphi_k\}\neq\emptyset$ (or, as $x\in X$, simply
$f(x)\leq\varphi_k$), implying $\varphi_k\geq f^*$, or $X\cap\{x
\suchthat f(x)\leq\varphi_k\}=\emptyset$, to infer that
$\varphi_k<f^*$, see, e.g., \cite{CL02}. However, in the case of
(possibly) infeasible iterates, $f_k\leq\varphi_k$ does not
necessarily imply that $\varphi_k$ is too large. On the other hand,
viewing the ISA iterates $x^k$ as points of the ``blown-up'' feasible
set $\mathcal{B}^{\varepsilon_{k-1}}_{X}\define\{x \suchthat
x=y+z, y\in~X, \norm{z}\leq\varepsilon_{k-1}\}$, then
$\mathcal{B}^{\varepsilon_k}_{X}\cap\{x \suchthat
f(x)\leq\varphi_k\}=\emptyset$ also implies that $\varphi_k<f^*$,
since $X\subseteq\mathcal{B}^{\varepsilon_k}_{X}$.

In view of Theorem~\ref{thm:ISA_conv_under}, keeping $\varphi_k$
constant once we recognized that $\varphi_k<f^*$ ensures convergence
of $(f_k)$ to $f^*$ (in practice, it may nevertheless be desirable to
further improve the estimate $\varphi_k$ in order to avoid overly
large steps in the vicinity of the optimum). The associated case
$\mathcal{B}^{\varepsilon_k}_{X}\cap\{x \suchthat
f(x)\leq\varphi_k\}=\emptyset$ can be detected in practice,
see~\cite[Section~III.C]{CL02} for details in the case of a feasible
method; these results are extensible to the ISA framework with
appropriate modifications.

The other case, $\varphi_k\geq f^*$, could be detected, e.g., with the help of 
an estimate of the Lipschitz constant of $f$ (recall that every convex function is locally Lipschitz and useful
estimates should usually be available) and the distances to $X$ implied by
the projection accuracies.

In the literature, various schemes have been considered as update
rules for variable targets $(\varphi_k)$, see, e.g.,
\cite{BS81,KAC91,GK99,SCT00,NB01,CL02,K04,LS05}. In principle, 
many such rules could be straightforwardly used in, or adapted to, a 
variable target value ISA. 
%For instance, we
%could compose the target level as $\varphi_k=\bar{\varphi}_k-c_k$ with
%an upper bound $\bar{\varphi}_k$ on $f^*$ and a monotonically
%nonincreasing sequence $(c_k)$ as in~\cite{CL02}; for
%Lipschitz-continuous~$f$, we could, e.g., work with
%$\bar{\varphi}_0\define f(x^0)+L\,d_{X}(x^0)$ (or just $f(x^0)$, if
%$x^0\in X$) and
%$\bar{\varphi}_{k+1}\define\min\{f(x^{k+1})+L\,\varepsilon_k,\bar{\varphi}_k\}$. In
%principle, we could keep $c_k\equiv c>0$ constant for all $k$, aiming
%at reaching some iteration $K$ in which $\varphi_K<f^*$ is recognized
%(existence of such a $K$ can be proven similarly to~\cite{CL02}), so
%that we could freeze the current target level and invoke
%Theorem~\ref{thm:ISA_conv_under} to achieve convergence. %As mentioned
%before, it may be more desirable in practice to decrease $c_k$ in this
%case.

\subsection{Computable bounds on $\d_{\X^*}(x^k)$}

The results in Theorems~\ref{thm:ISA_conv_over}
and~\ref{thm:ISA_conv_under} hinge on bounds
$\overline{\varepsilon}_k$ and $\tilde{\varepsilon}_k$ on the
projection accuracy parameters $\varepsilon_k$, respectively. These
bounds depend on unknown information and therefore seem of little
practical use such as, for instance, an automated accuracy control in
an implementation of the dynamic step size ISA. While the quantity
$f^*$ can sometimes be replaced by estimates directly, it will
generally be hard to obtain useful estimates for the distance of the
current iterate to the optimal set. However, such estimates are
available for certain classes of objective functions. We will %roughly
sketch several examples in the following.

For instance, when $f$ is a \emph{strongly convex function}, i.e.,
there exists some constant $C > 0$ such that for all $x,y$ and $\mu \in
[0,1]$
\[
f(\mu x+(1-\mu)y)~\leq~\mu f(x)+(1-\mu)f(y) - C\,\mu(1-\mu) \norm{x-y}^2,
\]
one can use the following upper bound on the distance to the optimal
set~\cite{KAC91}:
\[
\d_{X^*}(x)~\leq~\min\,\Big\{\,\sqrt{\tfrac{f(x)-f^*}{C}},\,\tfrac{1}{2C}\,\min_{h\in\partial f(x)}\,\norm{h}\,\Big\}.
\]

For functions $f$ such that $f(x)\geq C\, \norm{x} - D$, with
constants $C,D>0$, one can make use of $\d_{X^*}(x)\leq \norm{x} +
\tfrac{1}{C}(f^*+D)$, obtained by simply employing the triangle
inequality. Another related example class is induced by coercive
self-adjoint operators $F$, i.e., $f(x) \define \langle Fx,x\rangle
\geq C \norm{x}^2$ with some constant $C > 0$ and a scalar product
$\langle\cdot,\cdot\rangle$.  The (usually) unknown $f^*$ appearing
above may again be treated using estimates.

Yet another important class is comprised of functions which have a set
of weak sharp minima \cite{F88} over $X$, i.e., there exists a
constant $\mu>0$ such that
\begin{equation}\label{eq:weaksharp}
  f(x) - f^*~\geq~\mu\,\d_{\X^*}(x)\qquad\forall\,x\in \X.
\end{equation}
Using $\d_{\X^*}(x)\leq\d_{\X}(x)+\d_{\X^*}(\Pr_{\X}(x))$ for
$x\in\R^n$, we can then estimate the distance of $x$ to $X^*$ via the
weak sharp minima property of $f$. An important subclass of such
functions is composed of the polyhedral functions, i.e., $f$ has the
form $f(x)=\max\{\,a_i^\top x+b_i \suchthat 1\leq i\leq N\,\}$, where
$a_i\neq 0$ for all $i$; the scalar~$\mu$ is then given by
$\mu=\min\{\,\norm{a_i}\,\vert\,1\leq i\leq N\,\}$. Rephrasing
(\ref{eq:weaksharp}) as
\[
\d_{\X^*}(x)~\leq~\frac{f(x)-f^*}{\mu}\qquad\forall x\in\X,
\]
we see that for $\varphi\leq f^*$ (e.g., dual lower bounds $\varphi$),
\[
\d_{\X^*}(x)~\leq~\frac{f(x)-\varphi}{\mu}\qquad\forall x\in\X.
\]
Thus, when the bounds on the distance to the optimal set derived from
using the above inequalities become too conservative (i.e., too large,
resulting in very small $\tilde{\varepsilon}_k$-bounds), one could try
to improve the above bounds by improving the lower bound $\varphi$.

In practice on might have access to (problem-specific) estimates of
$\d_{\X^*}(x)$; in~\cite{CL02}, it is claimed that ``for most
problems'' prior experience or heuristical considerations can be used
to that end. For instance, if $X$ is compact, the diameter of $X$
leads to the (conservative) estimate
$\d_{X^*}(x)\leq\text{diam}(X)+\d_{X}(x)$.

\section{Examples}
\label{sec:examples}

In this section, we briefly discuss two examples in which we can design
adaptive approximate projections as considered in the ISA framework.
In the first example, we focus on the theoretical aspects of how our
notion of adaptive approximate projection could be used to handle a
certain class of constraints appearing in stochastic programs. The
second application considers a (deterministic) optimization problem
for which we specialize ISA and present some numerical
experiments.

\subsection{Convex expected value constraints}\label{sec:ChanceConstraints}

We consider \emph{expected value constraints}~\cite{P73,K09} of the following form
\begin{equation}\label{eq:chanceConstr}
  g(x)\define\E[f(x;\omega)]=\int_{\Omega} f(x;\omega)\,p(\omega)\,d\omega \leq\eta,
\end{equation}
where $\E$ denotes the expected value, $\omega\in\Omega\subseteq\R^q$
is a vector of random variables with density $p$, $x$ are
deterministic variables in $\R^n$, $f: \R^n \times \R^q
\rightarrow \R$, and $\eta \in \R$. If~$f$ is convex in $x$ for every $\omega\in\Omega$,
\eqref{eq:chanceConstr} is a convex constraint.
Expected value constraint appear in stochastic programming as, for instance, the
expectational form of chance constraints, see, e.g., \cite{CC59,BL99}, or when
modeling expected loss or Value-at-Risk via integrated chance
constraints, see, e.g., \cite{KH86,KM05,KHvdV06}.

While generally $g(x)$ cannot be easily computed exactly, it can be
approximated using Monte Carlo methods, if samples of $\omega$ can be
(cheaply) generated. Here, taking $M$ independent samples $\omega^1, \dots,
\omega^M$, yields the approximation
\begin{equation}\label{eq:g_approx}
  \hat{g}_M(x)\define\frac{1}{M}\sum_{i=1}^{M}f(x;\omega^i)
\end{equation}
of $g(x)$. Moreover, we assume that we can compute a subgradient
$G(x;\omega) \in \partial_x f(x;\omega)$ for each value of~$x$ and~$\omega$.
Thus, we have $h\define\E[G(x;\omega)] \in \partial g(x)$. We then use
the approximation
\begin{equation}\label{eq:h_approx}
  \hat{h}_M(x) \define \frac{1}{M}\sum_{i=1}^{M} G(x; \omega^i),
\end{equation}
which is a ``noisy unbiased subgradient'' of $g$ at $x$;
see~\cite{BM07} for details. 

Considering the Lagrangean $L(y,\lambda) = \tfrac{1}{2} \norm{x-y}^2 +
\lambda\, (g(y) - \eta)$ of the projection problem for some point $x$ and
the set of feasible points w.r.t.~\eqref{eq:chanceConstr}, the
optimality conditions for the projection obtained by differentiating $L$ are
\begin{align}
\label{eq:CClagrOC1} -x+y+\lambda\, h& =0,\quad \text{for some }h\in\partial g(y),\\
\label{eq:CClagrOC2} g(y)-\eta&=0.
\end{align}
Then, the idea is to replace $g(y)$ and $h$ by the estimates
$\hat{g}_M(y)$ and $\hat{h}_M(y)$, respectively. An
adaptive approximate projection is obtained by solving
\begin{equation}
\label{eq:CClagrApprox} y = x-\lambda\, \hat{h}_M(y), \qquad \hat{g}_M(y) = \eta.
\end{equation}
For an appropriate sampling process, we can adaptively keep control on
the resulting projection error (with high probability).

We now demonstrate this approach on a simple example constraint in
which the above system can be solved easily and we obtain explicit
projection error bounds: Consider a linear function with random
coefficients, i.e., $f(x;\omega) = \omega^\top x$ and $q = n$. This
particular type of constraint is closely related to integrated chance
constraints which are used, for instance, to model bounds on expected
losses of some sort; see, e.g., \cite{KH86,KM05}. For this choice of
$f$, our Monte Carlo estimates are
\begin{equation}\label{eq:MClinear}
  \hat{h}_M(x) = \hat{h}_M = \frac{1}{M}\sum_{i=1}^{M} \omega^i\qquad\text{and}\qquad \hat{g}_M(x)=\hat{h}_M^\top x.
\end{equation}
Note that if $\E[\hat{h}_M(x)]$ is unknown, the feasibility operator
construction in~\cite{NDP09} is not applicable. Moreover, assuming
$h$, $\hat{h}_M\neq 0$ corresponds to imposing a lower bound on the
subgradient norm, like in the convergence theorems for ISA. Observing
that $\hat{h}_M$ is independent of $x$ (so in particular,
$\hat{h}_M(y)=\hat{h}_M$ as well), we can solve
\eqref{eq:CClagrApprox} to obtain the
solution
\begin{equation}\label{eq:CClinearIPr}
  \Pr^M(x) \define x - \left(\frac{\hat{h}_M^\top x - \eta}{\norm{\hat{h}_M}^2}\right)\hat{h}_M
\end{equation}
to the approximated projection problem. The exact projection is given
by
\begin{equation}\label{eq:CClinearExPr}
  \Pr^\infty(x) \define x - \frac{h^\top x-\eta}{\norm{h}^2}h,
\end{equation}
and---as the notation suggests---we have $\Pr^\infty(x) =
\lim_{M\to\infty} \Pr^M(x)$ almost-surely, since
$\text{Prob}(\lim_{M\to\infty}\hat{h}_M=h)=1$ by the (strong) law of
large
numbers. %i.e., $\text{Prob}(\lim_{M\to\infty}\Pr^M(x)=\Pr^{\infty}(x))=1$.

For sufficiently large $M$, we can use explicit $(1-\alpha)$-confidence
intervals for the expected value $h=\E[\hat{h}_M]$ via the central
limit theorem, and eventually obtain
\begin{equation}\label{eq:CCIPrOp}
\text{Prob}\big(~\norm{\Pr^M(x)-\Pr^\infty(x)}\leq \varepsilon_M~\big) = 1-\alpha,
\end{equation}
where
\[
\varepsilon_M\define \left\lVert \frac{\hat{h}_M^\top x-\eta}{\norm{\hat{h}_M}^2}\hat{h}_M - \frac{\hat{h}_M^\top x-\eta+\overline{c}\cdot q_M^\top x}{\norm{\hat{h}_M+\overline{c}\cdot q_M}^2}(\hat{h}_M+\overline{c}\cdot q_M)\right\rVert,
\]
with $\overline{c}=-\text{sign}(\hat{h}_M^\top q_M)$ and
\[
q_M=\frac{q_{(1-\alpha/2)}}{\sqrt{M}\sqrt{M-1}}\left(\sqrt{\sum_{i=1}^{M}((\omega^i)_1-(\hat{h}_M)_1)^2},\dots,\sqrt{\sum_{i=1}^{M}((\omega^i)_n-(\hat{h}_M)_n)^2}\right)^\top,
\]
where $q_{(1-\alpha/2)}$ denotes the $(1-\tfrac{\alpha}{2})$-quantile
of the standard normal distribution. Thus, for any given
$\alpha\in(0,1)$ and for sufficiently large $M$, $\Pr^M$ defines an
adaptive approximate projection operator as specified in the ISA
framework, with probability $1-\alpha$.

It is noteworthy that the projection accuracy directly depends on $M$,
and in the linear example above we could iteratively refine the
estimate $\hat{h}_M$ easily by incorporating newly drawn independent
samples.

%An implementation of the described way to handle the above constraint
%types would be straightforward; however, we do not include
%computational results here because to create a meaningful test problem
%we would first need to introduce a lot of machinery from stochastic
%programming, Value-at-Risk computations etc, which goes beyond the
%scope of this paper.

\subsection{Compressed sensing}\label{sec:CompressedSensing}

Compressed Sensing (CS) is a recent and very active research field
dealing, loosely speaking, with the recovery of signals from
incomplete measurements. We refer the interested reader
to~\cite{D06,BDE09,CSweb} for more information, surveys, and key
literature. A core problem of CS is finding the sparsest solution to
an underdetermined linear system, i.e.,
\begin{equation}\label{eq:P0}
  \min\,\norm{x}_0\quad\text{s.\,t.}\quad Ax=b,\qquad\quad(A\in\R^{m\times n},~\text{rank}(A)=m,~m<n),
\end{equation}
where $\norm{x}_0$ denotes the $\ell_0$ quasi-norm or support size of
the vector $x$, i.e., the number of its nonzero entries. This problem
is known to be $\mathcal{NP}$-hard. Hence, a common approach is
considering the convex relaxation known as $\ell_1$-minimization or
Basis Pursuit \cite{CDS98}:
\begin{equation}\label{eq:P1}
  \min\,\norm{x}_1\quad\text{s.\,t.}\quad Ax=b.
\end{equation}
It was shown that under certain conditions, the solutions
of~\eqref{eq:P1} and~\eqref{eq:P0} coincide, see, e.g.,
\cite{CRT06,D06}. This motivated a large amount of research on the
efficient solution of~\eqref{eq:P1}, especially in large-scale
settings. In this section, we briefly outline a specialization of the
ISA to the $\ell_1$-minimization problem~\eqref{eq:P1} and present
some numerical experiments indicating that the algorithm is an
interesting candidate in the context of Compressed
Sensing. %For a detailed discussion and an extensive computational comparison of various $\ell_1$-solvers, see \cite{LPT11}.

\paragraph{Subgradients.}
The subdifferential of the $\ell_1$-norm at a point $x$ is given by
\begin{equation}\label{eq:L1subdiff}
  \partial\norm{x}_1 = \Big\{~h\in
  [-\mathds{1},\mathds{1}]^n~\Big\vert~h_i=\frac{x_i}{\vert x_i\vert},\quad
  \forall\, i\in\{1,\dots,n\}\text{ with } x_i \neq 0~\Big\}.
\end{equation}
We may therefore simply use the signs of the iterates as subgradients,
i.e.,
\begin{equation}\label{eq:L1subgrad}
  \partial\norm{x^k}_1 \ni h^k~\define ~\text{sign}(x^k)~=~
  \begin{cases} 
    \hfill 1, & (x^k)_i>0,\\ 
    \hfill 0, &(x^k)_i=0,\\ 
    ~-1, & (x^k)_i<0.
  \end{cases}
\end{equation}
As long as $b \neq 0$, the upper and lower bounds on the norms of the
subgradients satisfy $\underline{H} \geq 1$ and $\overline{H} \leq n$.

\paragraph{Adaptive approximate projection.}

For linear equality constraints as in~\eqref{eq:P1}, the Euclidean
projection of a point $z\in\R^n$ onto the affine feasible set
$\X\define \{\,x \suchthat Ax=b\,\}$ can be explicitly calculated as
\begin{equation}\label{eq:L1proj}
  \Pr_{\X}(z)~=~\big(I-A^\top (AA^\top)^{-1}A\big)z+A^\top (AA^\top)^{-1}\,b,
\end{equation}
where $I$ denotes the ($n\times n$) identity matrix. However, for
numerical stability, we wish to avoid the explicit calculation of the
projection matrix because it involves determining the inverse of the matrix
product $AA^\top$. Instead of applying (\ref{eq:L1proj}) in each
iteration, we can use the following adaptive procedure:
\begin{align}
  \label{eq:L1projIter1}&\qquad z^k\define x^k-\alpha_k h^k\qquad\text{(unprojected next iterate)},\\
  \label{eq:L1projIter2}&\qquad\text{find an approximate solution }q^k\text{ of }AA^\top q=Az^k-b,\\
  \label{eq:L1projIter3}&\qquad x^{k+1}\define z^k-A^\top q^k.
\end{align}
Note that the matrix $AA^\top$ is symmetric and positive definite, for
$A$ with full (row-)rank $m$. Hence, the linear system in
(\ref{eq:L1projIter2}) can be solved by an iterative method, e.g., the
method of Conjugate Gradients (CG) \cite{HS52}. % In exact arithmetic,
For a given $\varepsilon_k$, stopping the CG procedure in
(\ref{eq:L1projIter2}) as soon as the iteratively updated approximate
solution $q^k$ satisfies
\begin{equation}\label{eq:CS_cg_stop}
  %\norm{r_q^k}_2~\leq~\sigma_{\min}(A)\,\varepsilon_k,
  \norm{AA^\top q^k-\big(A(x^k-\alpha_k h^k)-b\big)}_2~\leq~\sigma_{\min}(A)\,\varepsilon_k,
\end{equation}
where $\sigma_{\min}(A)>0$ is the smallest singular value of $A$,
ensures that \eqref{eq:L1projIter1}--\eqref{eq:L1projIter3} form an
adaptive approximate projection operator of the type
\eqref{eq:IPr}. %, see~\cite{LPT11}.
Note that a truncated CG procedure (with any fixed number of
iterations) can also be shown to define a ``feasibility operator'' of
the type considered in~\cite{NDP09}.

Furthermore, to obtain computable upper bounds on $(\varepsilon_k)$,
we can use the results about weak sharp minima discussed in the
previous section: The \l-norm can be rewritten as a polyhedral
function. With $\varphi\leq f^*$ (which is easily available, e.g.,
$\varphi=0$), we can thus derive
\[
\d_{\X^*}(x^k)~\leq~2\frac{\norm{Ax^k-b}_2}{\sigma_{\min}(A)}+\frac{\norm{x^k}_1-\varphi}{\sqrt{n}}.
\]
In total, this yields bounds that can be easily computed from the
original data only. Theorems~\ref{thm:ISA_conv_apriori},
\ref{thm:ISA_conv_over}, or \ref{thm:ISA_conv_under} then provide
explicit convergence statements.

\subsubsection{Numerical Experiments} 
It is well-known that~\eqref{eq:P1} can be solved as a linear program
(LP), e.g., employing the standard variable split $x=x^+-x^-$:
\begin{equation}\label{eq:P1LP}
  \min\,x^{+} +x^{-}\quad\text{s.\,t.}\quad Ax^{+} -Ax^{-} =b,~x^{+}\geq 0,~x^{-}\geq 0.
\end{equation}
Another common approach to~\eqref{eq:P1} is to solve a sequence of
regularized problems of the form
\begin{equation}\label{eq:Homo}
  \min\,\tfrac{1}{2}\norm{Ax-b}_2^2+\tau\norm{x}_1,
\end{equation}
with decreasing $\tau$. As $\tau\to 0$, the solution sequence
$x(\tau)$ of~\eqref{eq:Homo} converges to a solution
of~\eqref{eq:P1}. The homotopy method (see, e.g., \cite{OPT00,MCW05})
traces this solution path for decreasing $\tau$ and has the desirable
property to require only $k$ steps to reach the optimal solution~$x^*$
to~\eqref{eq:P1}, if $x^*$ has only $k$ nonzero entries and $k$ is
sufficiently small.

We performed experiments to compare our ISA
Algorithm~\ref{alg:DynamicISA}, applied to~\eqref{eq:P1} (using
adaptive approximate or exact projections), with the commercial
LP-solver \textsc{Cplex} 12.5 (dual simplex method applied
to~\eqref{eq:P1LP}) and the homotopy implementation (version 1.0)
available at \url{http://users.ece.gatech.edu/~sasif/homotopy/}. In
our ISA implementation we employ at most $5$ CG steps to approximate
the projection; albeit differing from theory, this turned out to
suffice. Moreover, the subgradients are stabilized as in~\cite{L98},
and the parameter $\lambda_k$ is halved after $5$ consecutive
iterations without relevant improvement of the objective
($\lambda_0=0.85$); the method terminates when the step sizes become
too small or if a stagnation of the algorithmic process is detected.
By stagnation, we mean that either the objective improvement stalls
over a span of $500$ iterations, or the approximate support
$S=\{i:\abs{x^k_i}>\max\{10^{-6},s\}\}$ does not change over $10$ successive
updates, which are performed every $m/100$ iterations; here $s$ is chosen such that the
entries $x^k_j$ with $\abs{x^k_j}\geq s$ account for at least
$99.99\%$ of $\norm{x^k}_1$. Finally, as a
postprocessing step after termination, we try to improve the solution
by solving the system restricted to columns indexed by $S$, similar to
the ``debiasing'' step described in~\cite[Section II.I]{WFN09}.

Note that in contrast to \textsc{Cplex}, the homotopy method and ISA
are implemented in \textsc{Matlab} (version R2012a/7.14). Moreover, by default,
\textsc{Cplex} ensures feasibility in the sense that the computed
solution $\overline{x}$ obeys $\norm{A\overline{x}-b}_\infty\leq
10^{-6}$; from the respective convergence results, both the homotopy
method and ISA will reach this level of feasibility after finitely
many iterations. As a safeguard, we added an additional high-accuracy
projection after regular termination. However, this step was not
required for the homotopy method, and only on a single instance for
ISA (this induced additional running time and the time for the
postprocessing step is incorporated in the times reported below).

The first test uses a $1024\times 4096$ Gaussian matrix, the second
one a partial discrete cosine transform (DCT) matrix consisting of
$512$ randomly drawn rows of the $2048\times 2048$ DCT matrix; all
columns are normalized to unit Euclidean length. For both matrices, we
constructed ten vectors $x^i$ with sparsities $\norm{x^i}_0=i\cdot
m/10$, $i\in\{1,\dots,10\}$, (rounded down to the next integer
value). The nonzero entries are $\pm 1$ and each $x^i$ is the known
unique solution to the instance given by the respective matrix $A$ and
right hand side vector $b\define Ax^i$, where uniqueness was achieved
by ensuring the ``strong source condition'' (see, e.g., \cite{GHS11})
by means of the methodology proposed in~\cite{KL13}.

\begin{figure}[t]
  \centering
  \subfigure[$\ell_2$-distances to optimum for instances with $1024\times 4096$ Gaussian matrix.]{\includegraphics[width=0.475\textwidth]{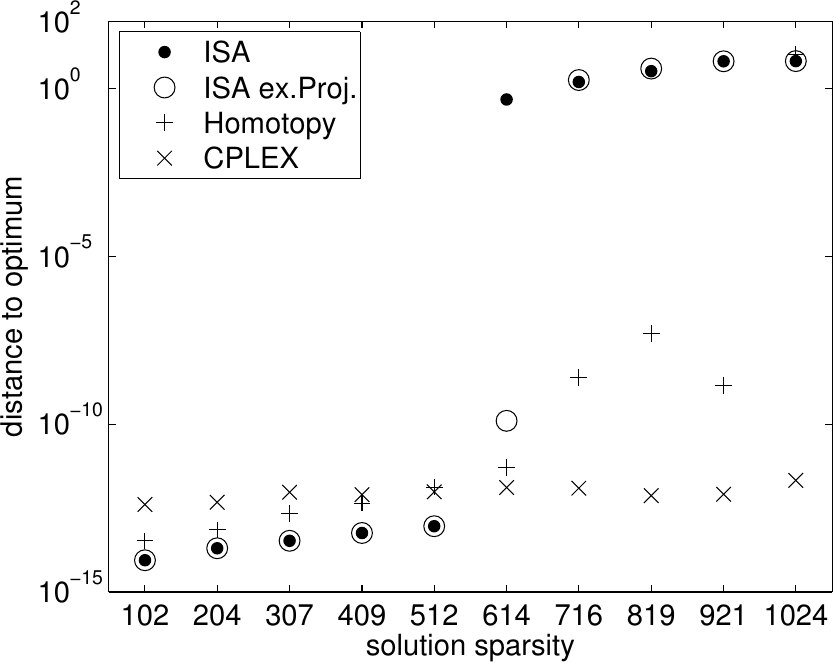}}
  \label{fig:GaussDist}
  %\centering  
  \hfill
  \subfigure[Running times (s) for instances with $1024\times 4096$ Gaussian matrix.]{\includegraphics[width=0.463\textwidth]{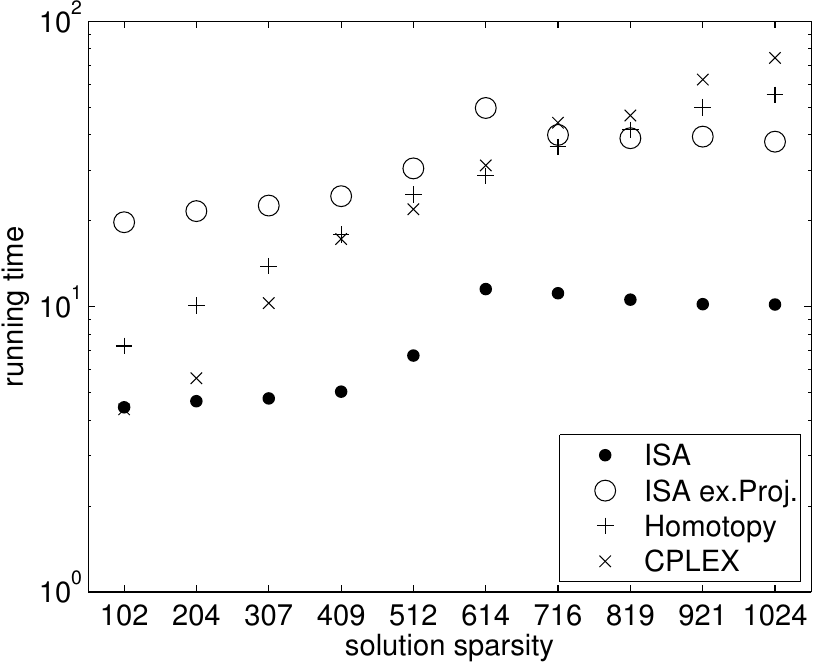}}  
  \label{fig:GaussTime}
  \subfigure[$\ell_2$-distances to optimum for instances with $512\times 2048$ partial DCT matrix.]{\includegraphics[width=0.475\textwidth]{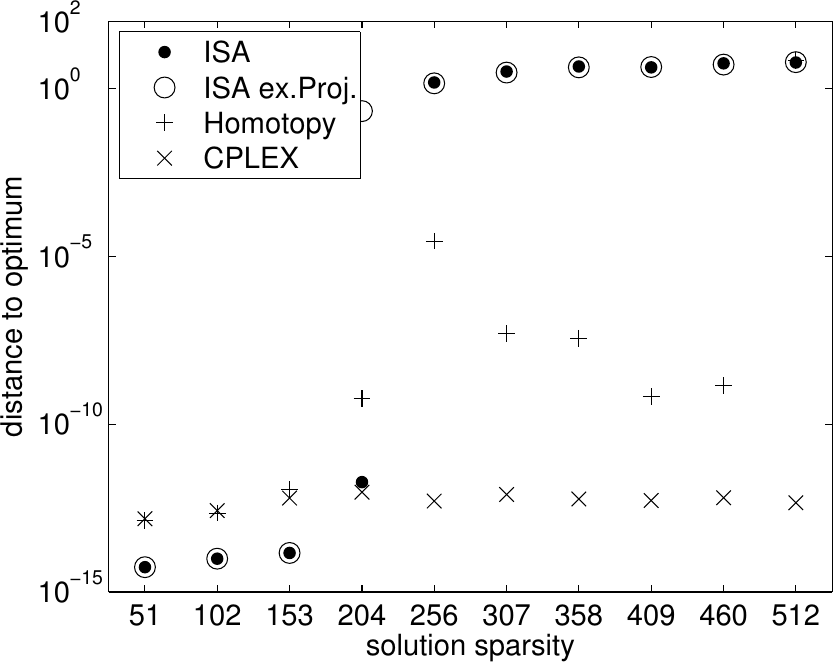}}
  \label{fig:partdctDist}
  %\centering  
  \hfill
  \subfigure[Running times (s) for instances with $512\times 2048$ partial DCT matrix.]{\includegraphics[width=0.47\textwidth]{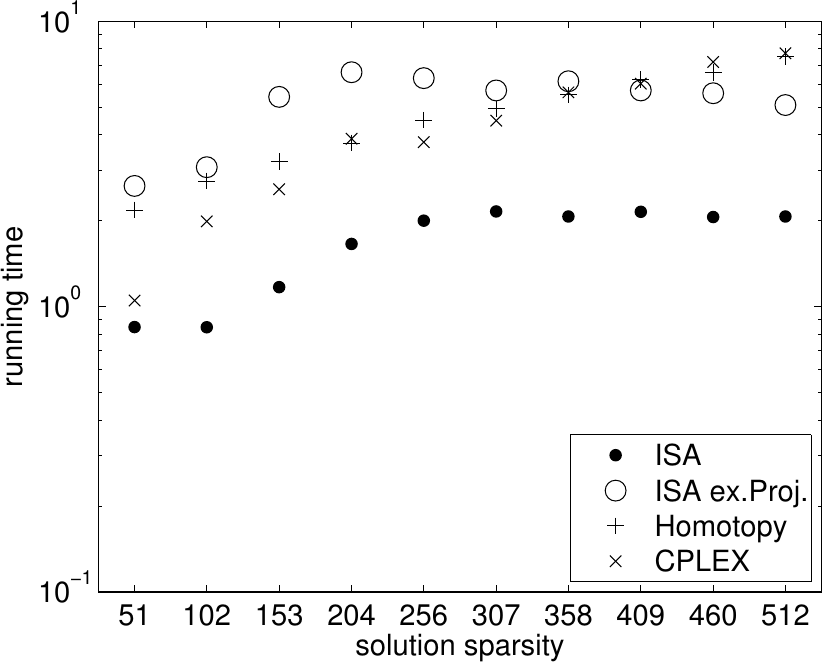}}
  \label{fig:partdctTime}
  \caption{Numerical experiments for Gaussian matrix ((a) and (b)) and partial DCT matrix ((c) and (d)), each with normalized columns, for varying solution sparsities.}
  \label{fig:experiments}
\end{figure}

Figure~\ref{fig:experiments} shows the running times (in seconds) and
the $\ell_2$-norm distances to the respective known optimal solution.
As explained above, all solutions are feasible to within an
$\ell_\infty$-tolerance of $10^{-6}$.
%;small values for the latter clearly imply small feasibility violations as well. 
The experiments show that using adaptive approximate projections
instead of the exact ones in ISA saves a considerable amount of time,
as was to be expected. The achieved final accuracy is almost always
(nearly) the same. For the varying sparsity levels of the solution, we
see that all solvers struggle when the number of nonzero entries in
the optimum exceeds about $m/2$: \textsc{Cplex} and the homotopy
method still produce mostly accurate solutions but at the cost of a
significant increase in the required solution times (note the
logarithmic scales on the vertical axes), ISA on the other hand has a
somewhat more stable runtime behavior, but loses accuracy when the
solution is dense.

Since in Compressed Sensing, the solutions encountered are typically
very sparse, the interesting cases are those with sparsity (much)
smaller than $m/2$. Clearly, for such sparse optimal solutions, ISA
(with adaptive approximate projections) is superior to \textsc{Cplex}
and the homotopy implementation both in terms of accuracy and
speed. Thus, these examples show the potential of ISA as a successful
algorithm for CS sparse recovery.

\section{Concluding remarks}

Several aspects remain subject to future research. For instance, it
would be interesting to investigate whether our framework extends
to %more general
(infinite-dimensional) Hilbert space settings, incremental subgradient
schemes, bundle methods (see, e.g., \cite{HUL93,K90}), or Nesterov's
algorithm~\cite{N05}. It is also of interest to consider how the ISA
framework could be combined with error-admitting settings such as
those in \cite{Z10,NB10}, i.e., for random or deterministic
(non-vanishing) noise and erroneous function or subgradient
evaluations. Some of the recent results in~\cite{NB10}, which all
require feasible iterates, seem conceptually close to our
convergence analyses, so we presume a blend of the two approaches to
be rather fruitful. It would also be of interest to investigate
convergence behavior with other general notions of ``adaptive
approximate projections'', e.g., solving the projection problem with
an approximation algorithm with additive or multiplicative performance
guarantee.

From a practical viewpoint, it will be interesting to see how ISA, or
possibly a variable target value variant as described in
Section~\ref{subsect:VTVM}, compares with other solvers in terms of
solution accuracy and runtime. For the $\ell_1$-minimization problem
(\ref{eq:P1}), we have seen in Section~\ref{sec:CompressedSensing}
that ISA promises to be an interesting candidate; an extensive
computational comparison of various state-of-the-art $\ell_1$-solvers,
including (a more refined version of) our ISA implementation, can be
found in~\cite{LPT11}.  An extensive test for convex expected value
constraints, while beyond the scope of this paper, would be an
interesting further line of work.

\paragraph{Acknowledgments.} We thank the anonymous referees for their
numerous helpful comments which greatly helped improving this paper.
% Delving deeper into stochastic programming with
% convex expected value constraints went beyond the scope of this paper,
% but since an implementation of the adaptive approximate projection
% approach introduced in Section~\ref{sec:ChanceConstraints} should be
% straightforward, investigating practical applications in relevant
% problems from that field of research---either as a ``stand-alone''
% tool, or as a subroutine of ISA---would be an interesting further
% line of work.

\bibliographystyle{siam} %splncs
\bibliography{Paper_ISA_Convergence}

\end{document}